\documentclass[mathpazo]{cicp}

\usepackage{caption}
\usepackage{url}
\usepackage{subfigure}
\usepackage{color}
\usepackage{booktabs}
\usepackage{mathtools, latexsym}
\usepackage[numbers]{natbib}
\usepackage{psfrag,epsfig,latexsym,,amscd,url }
\usepackage{bm}
\usepackage{appendix}
\usepackage{epstopdf}
\allowdisplaybreaks[4]

\begin{document}
\title{Finite Difference Approximation with ADI Scheme for Two-dimensional Keller-Segel Equations}

\author[X.~Li]{Yubin Lu, Chi-An Chen, Xiaofan Li\corrauth, Chun Liu}
\address{Department of Applied Mathematics, College of Computing, Illinois Institute of Technology, Chicago, IL 60616, USA}
\email{{\tt lix@iit.edu} (X.~Li)}


\begin{abstract}
Keller-Segel systems are a set of nonlinear partial differential equations used to model chemotaxis in biology. In this paper, we propose two alternating direction implicit (ADI) schemes to solve the 2D Keller-Segel systems directly with minimal computational cost, while preserving positivity, energy dissipation law and mass conservation. One scheme unconditionally preserves positivity, while the other does so conditionally. Both schemes achieve second-order accuracy in space, with the former being first-order accuracy in time and the latter second-order accuracy in time. Besides, the former scheme preserves the energy dissipation law asymptotically. We validate these results through numerical experiments, and also compare the efficiency of our schemes with the standard five-point scheme, demonstrating that our approaches effectively reduce computational costs.
\end{abstract}


\maketitle

\section{Introduction}
The Keller-Segel system as a mathematical model was proposed to model the phenomenon of chemotaxis in biology, established by Patlak \cite{Patlak} and Keller and Segel \cite{Keller-Segel,Keller-Segel1}. Chemotaxis is the migration of organisms in response to chemical signals. Usually, the organisms are attracted by chemical signals that could be emitted by themselves. On the other hand, there is another hidden mechanism, namely the diffusion phenomenon. These two mechanisms compete with each other, leading to different interesting biological phenomena. These phenomena could be summarized by the Keller-Segel equations, a system of parabolic-parabolic equations or parabolic-elliptic equations in a limiting case. So far, the Keller-Segel equations are not limited to describing such biological processes, but can also be widely used to describe various phenomena in social sciences and other competitive systems.\\
One of the most important mathematical issues of the Keller-Segel system is the well-posedness of their solutions. This issue has been well studied by mathematicians in some sense \cite{review1,review2,review3,blow-up1}. Roughly speaking, if the initial mass is less than a critical value, the solution exists globally, otherwise the solution will blow up in finite time. Certainly, blow-up does not occur in real scenarios. This indicates that the classical Keller-Segel model has limitations. Therefore, several improved models are designed to relax such restrictions \cite{Booktransport}. We will focus on the classical one.\\
In addition to the analytic aspect of the Keller-Segel equations, the numerical aspect is also extremely important for real applications. The Keller-Segel system consists of two equations, one of which describes the evolution of density of organisms and the other describes the evolution of the chemoattractant concentration. Naturally, the density and the concentration should remain nonnegative, the density should also preserve mass conservation under suitable boundary conditions, and the energy dissipation law holds.  A question arises: how do we design numerical solvers that preserve these properties? In addressing this question, many works have been proposed to solve such type equations with specific conditions. The most popular methods including linear/nonlinear finite volume schemes \cite{finite-volume1,finite-volume2,finite-volume3,finite-volume4,finite-volume5,finite-volume6}, discontinuous Galerkin schemes \cite{DiscontinuousGalerkin1,DiscontinuousGalerkin2}, discontinuous finite element scheme \cite{DiscontinuousFiniteElement}, finite difference scheme \cite{Finitedifference}, spectral method \cite{spectral}, and others \cite{SVA,hybrid}. Moreover, several general ways have been proposed to preserve the energy dissipation law for specific problems, including scalar auxiliary variable approach (SAV) \cite{SAV}, convex splitting method \cite{convexsplitting}, stabilization method \cite{stabilization1,stabilization2}, and the method of invariant energy quadratization (IEQ) \cite{IEQ1,IEQ2}. Furthermore, numerous numerical schemes have been developed to address the challenges posed by the Poisson-Nernst-Planck equations. Given the interrelated nature of Keller-Segel equations and Poisson-Nernst-Planck equations in the context of numerical scheme design, we recommend that readers consult \cite{LiuC2023,LiuHL2021,LiuHL2023,ShenJ2021}, and the references cited therein for more insights. Although these numerical methods are developed to preserve positivity, mass conservation and energy dissipation and so on, very few of the existing papers take into account the reduction of computational cost while still preserving these desired properties. 

In this work, we consider the following 2D Keller-Segel equations
\begin{align}
    \label{eqn:original_rho} \partial_t\rho&=\Delta\rho-\nabla\cdot(\rho\nabla c), \\
\label{eqn:original_c} \varepsilon\partial_tc&=\Delta c+\rho,\\
    \rho(\bold{x},0)&=f(\bold{x}),\qquad c(\bold{x},0)=g(\bold{x}),\nonumber
\end{align}
where $\rho$ and $c$ are the cell density and the chemoattractant concentration respectively, $\bold{x}=(x,y)\in\Omega\subset\mathbb{R}^2$, and $\varepsilon$ is a nonnegative constant. To preserve positivity,  Liu {\it{et al}}. \cite{Finitedifference} reformulated the density equation \eqref{eqn:original_rho} as the following symmetric form:
\begin{equation}\label{eqn:symmtric_form}
    \partial_t\rho=\nabla\cdot\left(e^{c}\nabla\left(\frac{\rho}{e^{c}}\right)\right),
\end{equation}
and proposed the following first-order semi-discrete scheme:
\begin{align}
    \label{eqn:semi-discrete_rho} \frac{\rho^{n+1}-\rho^{n}}{\Delta t} &= \Delta\rho^{n+1}-\nabla\cdot(\rho^{n+1}\nabla c^{n+1}),\\
    \label{eqn:semi-discrete_c}\varepsilon\frac{c^{n+1}-c^{n}}{\Delta t}&=\Delta c^{n+1}+\rho^n.
\end{align}
This scheme is not only stable, but also preserves positivity and mass conservation under some suitable conditions. See \cite{Finitedifference} for more details.\\
It is not hard to realize that naive spatial discretizations of \eqref{eqn:semi-discrete_rho} might destroy the positivity of the solution  and trigger instability. Denoting $M^{n+1}=e^{c^{n+1}}$, Liu {\it{et al}}. \cite{Finitedifference} rewrote the semi-discrete scheme \eqref{eqn:semi-discrete_rho} based on the symmetric form of the density equation \eqref{eqn:symmtric_form}:
\begin{align}
\label{eqn:semi-symmtric_form_rho} \frac{\rho^{n+1}-\rho^{n}}{\Delta t} &= \nabla\cdot\left(M^{n+1}\nabla\left(\frac{\rho^{n+1}}{M^{n+1}}\right)\right),\\
\label{eqn:semi-symmtric_form_c} \varepsilon\frac{c^{n+1}-c^{n}}{\Delta t}&=\Delta c^{n+1}+\rho^n.
\end{align}
Furthermore, denoting $h^{n+1}=\frac{\rho^{n+1}}{\sqrt{M^{n+1}}}$, one can reformulate \eqref{eqn:semi-symmtric_form_rho} into
\begin{align}
    h^{n+1}-\frac{\Delta t}{\sqrt{M^{n+1}}}\nabla\cdot\left(M^{n+1}\nabla\left(\frac{h^{n+1}}{\sqrt{M^{n+1}}}\right)\right)=\frac{\rho^{n}}{\sqrt{M^{n+1}}}.
\end{align}
We provide some finite difference notations here for convenience
\begin{align}
    \Delta_{+t}v(x,y,t)&:=v(x,y,t+\Delta t)-v(x,y,t),\\
    \delta_x v(x,y,t)&:=v(x+\frac{1}{2}\Delta x,y,t)-v(x-\frac{1}{2}\Delta x,y,t), \\
    \delta_{x}^{2}v(x,y,t)&:=v(x+\Delta x,y,t)-2v(x,y,t)+v(x-\Delta x,y,t),\\
     \delta_y v(x,y,t)&:=v(x,y+\frac{1}{2}\Delta y,t)-v(x,y-\frac{1}{2}\Delta y,t), \\
    \delta_{y}^{2}v(x,y,t)&:=v(x,y+\Delta y,t)-2v(x,y,t)+v(x,y-\Delta y,t),
\end{align}
\begin{align}
    \label{eqn:ADI_rho_x}\tau_x v&:=\frac{1}{\sqrt{M^{n+1}}}\delta_x\left( M^{n+1}\delta_x\left(\frac{v}{\sqrt{M^{n+1}}}\right)\right),\\
    \label{eqn:ADI_rho_y}\tau_y v&:=\frac{1}{\sqrt{M^{n+1}}}\delta_y\left( M^{n+1}\delta_y\left(\frac{v}{\sqrt{M^{n+1}}}\right)\right),
\end{align}
\begin{align}
    \label{eqn:ADI_rho_x_middle}\bar{\tau}_x v&:=\frac{1}{\sqrt{M^{n+\frac{1}{2}}}}\delta_x\left( M^{n+\frac{1}{2}}\delta_x\left(\frac{v}{\sqrt{M^{n+\frac{1}{2}}}}\right)\right),\\
    \label{eqn:ADI_rho_y_middle}\bar{\tau}_y v&:=\frac{1}{\sqrt{M^{n+\frac{1}{2}}}}\delta_y\left( M^{n+\frac{1}{2}}\delta_y\left(\frac{v}{\sqrt{M^{n+\frac{1}{2}}}}\right)\right).
\end{align}
We denote the spatial domain of solutions by $\Omega=[a,b]\times[c,d]$, and the mesh sizes by $\Delta x$ and $\Delta y$ in $x$- and $y$-directions respectively. We use $c_{i,j}^{n}$ to denote the numerical solution for $c(x_i,y_j,t_n)$. The other variables could be defined in similar ways.\\
Applying the standard five-point method for spatial discretization, one can obtain
\begin{align}
\label{eqn:semidiscrete_c} \left(\frac{\varepsilon}{\Delta t}-\frac{1}{(\Delta x)^2}\delta_{x}^{2}-\frac{1}{(\Delta y)^2}\delta_{y}^{2}\right)c_{i,j}^{n+1}&=\frac{\varepsilon}{\Delta t}c_{i,j}^{n}+\rho_{i,j}^{n},\\
\label{eqn:semidiscrete_rho} \left(1-\frac{\Delta t}{(\Delta x)^2}\tau_x-\frac{\Delta t}{(\Delta y)^2}\tau_y\right)h_{i,j}^{n+1}&=\frac{\rho_{i,j}^{n}}{\sqrt{M_{i,j}^{n+1}}}.
\end{align}
At $n$-th time step, one solves a sparse, $N_xN_y\times N_xN_y$ linear system for $c^{n+1}$ from \eqref{eqn:semidiscrete_c}; then solves another sparse, $N_xN_y\times N_xN_y$ linear system for $h^{n+1}$ from \eqref{eqn:semidiscrete_rho}. Subsequently, we can obtain the density $\rho_{i,j}^{n+1}$ using the following formula:
\begin{equation}
    \rho_{i,j}^{n+1}=h_{i,j}^{n+1}\sqrt{M_{i,j}^{n+1}}.
\end{equation}
This scheme not only preserves positivity and mass conservation but also decouples a nonlinear system into a two-stage linear system. This means that the Keller-Segel system could be solved efficiently using iterative methods. However, such as in \cite{Finitedifference}, if the grid is sufficiently fine, the computational cost becomes high. Especially, the solution to the Keller-Segel system blows up in finite time when the initial mass is large than  a critical value \cite{Booktransport}. In this case, usually, we would have to use an extremely fine grids to capture more precise asymptotic behavior and seek most efficient scheme to minimize computational cost. In this work, we construct an alternating direction implicit (ADI) scheme so that the linear systems corresponding to the discretized equations corresponding to \eqref{eqn:semi-symmtric_form_rho} and \eqref{eqn:semi-symmtric_form_c} can be solved directly with minimal computational cost while still preserving positivity, mass conservation and energy dissipation law asymptotically.


\section{ADI schemes}
ADI schemes solve the sparse linear systems using fast direct methods when parabolic partial differential equations are discretized in 2D or higher dimensions. In this section, we propose an ADI scheme for solving the Keller-Segel system \eqref{eqn:original_rho} and \eqref{eqn:original_c}. The scheme preserves positivity, mass conservation and energy dissipation law asymptotically. \\

\subsection{The ADI scheme}
\textbf{Concentration equation.} The five-point scheme \eqref{eqn:semidiscrete_c} for the semi-discrete scheme \eqref{eqn:semi-discrete_c} can be written as follows 
\begin{align}\label{eqn:semi-discrete4ADI_c1}
    (1-\mu_{x}^{\varepsilon}\delta_{x}^{2}-\mu_{y}^{\varepsilon}\delta_{y}^{2})c_{i,j}^{n+1}=c_{i,j}^{n} + \frac{\Delta t}{\varepsilon}\rho_{i,j}^{n},\quad\text{where}\quad\mu_{x}^{\varepsilon}=\frac{\Delta t}{\varepsilon(\Delta x)^2}\quad\text{and}\quad\mu_{y}^{\varepsilon}=\frac{\Delta t}{\varepsilon(\Delta y)^2}.
\end{align}
 Instead of \eqref{eqn:semi-discrete4ADI_c1}, we propose an ADI scheme discretizing Eq. \eqref{eqn:semi-discrete_c} as follows
\begin{align}\label{eqn:semi-discrete4ADI_c2}
    (1-\mu_{x}^{\varepsilon}\delta_{x}^{2})(1-\mu_{y}^{\varepsilon}\delta_{y}^{2})c_{i,j}^{n+1}=c_{i,j}^{n}+\frac{\Delta t}{\varepsilon}\rho_{i,j}^{n},
\end{align}
which is equivalent to the following 2-step method:
\begin{align}
    \label{eqn:semi-discrete4ADI_c3}(1-\mu_{x}^{\varepsilon}\delta_{x}^{2})c_{i,j}^{*}&=c_{i,j}^{n}+\frac{\Delta t}{\varepsilon}\rho_{i,j}^{n},\\
    \label{eqn:semi-discrete4ADI_c4}(1-\mu_{y}^{\varepsilon}\delta_{y}^{2})c_{i,j}^{n+1}&=c_{i,j}^{*}.
\end{align}
\textbf{Density equation.}
The five-point scheme \eqref{eqn:semidiscrete_rho} for the density $\rho$ can be rewritten as
\begin{align}\label{eqn:semi-discrete4ADI_rho1}
    [1-({\mu}_{x}\tau_x+{\mu}_{y}\tau_y)]h_{i,j}^{n+1}=\frac{\rho_{i,j}^{n}}{\sqrt{M_{i,j}^{n+1}}}, \quad\text{where}\quad{\mu}_x=\frac{\Delta t}{(\Delta x)^2}, {\mu}_y=\frac{\Delta t}{(\Delta y)^2}.
\end{align}
$\tau_x$ and $\tau_y$ are operators defined as in \eqref{eqn:ADI_rho_x} and \eqref{eqn:ADI_rho_y}. Instead of \eqref{eqn:semi-discrete4ADI_rho1}, we consider the following ADI method for the semi-discrete equation \eqref{eqn:semi-discrete_rho}
\begin{align}
    \label{eqn:semi-discrete4ADI_rho11} (1-{\mu}_x\tau_x)h_{i,j}^{*}&=\frac{\rho_{i,j}^{n}}{\sqrt{M_{i,j}^{n+1}}},\\
    \label{eqn:semi-discrete4ADI_rho12}(1-{\mu}_y\tau_y)h_{i,j}^{n+1}&=h_{i,j}^{*}.
\end{align}
The 2-step scheme \eqref{eqn:semi-discrete4ADI_rho11} and \eqref{eqn:semi-discrete4ADI_rho12} is equivalent to
\begin{align}\label{eqn:semi-discrete4ADI_rho2}
    (1-{\mu}_x\tau_x)(1-{\mu}_y\tau_y)h_{i,j}^{n+1}=\frac{\rho_{i,j}^{n}}{\sqrt{M_{i,j}^{n+1}}}.
\end{align}
It is straight forward to verify that the leading-order truncation errors of the ADI scheme \eqref{eqn:semi-discrete4ADI_c2} and \eqref{eqn:semi-discrete4ADI_rho2} for Keller-Segel system are the same as the five-point scheme \eqref{eqn:semi-discrete4ADI_c1} and \eqref{eqn:semi-discrete4ADI_rho1}, i.e., second-order in space and first-order in time. See Appendix \ref{errors}. The stability of this scheme holds if the Keller-Segel equations satisfy a technical condition and a small data condition. See \cite{Finitedifference} for more details.
\subsection{Conservative properties of the ADI scheme}


\subsubsection{Positivity preservation}
Preserving positivity in numerical solutions of the Keller-Segel equations is crucial for obtaining physically meaningful results, as it ensures that the solutions accurately reflect the behavior of the system being modeled. In this subsection, we prove that the ADI schemes \eqref{eqn:semi-discrete4ADI_c2} and \eqref{eqn:semi-discrete4ADI_rho2} preserve positivity for the Keller-Segel equations.\\
\begin{theorem}\label{Thm1}
Suppose $\rho_{i,j}^{0}\geq 0$, then the ADI scheme \eqref{eqn:semi-discrete4ADI_rho2} preserves $\rho_{i,j}^{n}\geq 0, n\in\mathbb{N}$.\\
\begin{proof}
In order to prove the ADI scheme \eqref{eqn:semi-discrete4ADI_rho2} preserves posititvity, it is enough to prove the equivalent form \eqref{eqn:semi-discrete4ADI_rho11} and \eqref{eqn:semi-discrete4ADI_rho12} preserves positivity. Therefore, we prove the first equation \eqref{eqn:semi-discrete4ADI_rho11} preserves positivity. The positivity-preserving property of the second equation \eqref{eqn:semi-discrete4ADI_rho12} could be proved in a similar way. \\
Recalling the fully discrete scheme of the operator $\tau_x$:
\begin{align}
    \tau_x h_{i,j}^{*}&=\frac{1}{\sqrt{M_{i,j}^{n+1}}}\delta_x\left(M_{i,j}^{n+1}\delta_x\left(\frac{h^{*}}{\sqrt{M^{n+1}}}\right)_{i,j}\right)\nonumber\\
    &=\frac{1}{\sqrt{M_{i,j}^{n+1}}}\left[\sqrt{M_{i,j}^{n+1}M_{i+1,j}^{n+1}}\left(\frac{h_{i+1,j}^{*}}{\sqrt{M_{i+1,j}^{n+1}}}-\frac{h_{i,j}^{*}}{\sqrt{M_{i,j}^{n+1}}}\right) - \sqrt{M_{i-1,j}^{n+1}M_{i,j}^{n+1}}\left(\frac{h_{i,j}^{*}}{\sqrt{M_{i,j}^{n+1}}}-\frac{h_{i-1,j}^{*}}{\sqrt{M_{i-1,j}^{n+1}}}\right)\right].
\end{align}
Here, $M_{i-\frac{1}{2},j}^{n+1}$ is approximated by its geometric mean $\sqrt{M_{i-1,j}^{n+1}M_{i,j}^{n+1}}$ and similarly for the other terms at half grid points. Therefore, the first equation of the scheme for Eq. \eqref{eqn:semi-discrete4ADI_rho11} can be written as:
\begin{align}\label{eqn:proofPositivity}
    h_{i,j}^{*}&=\frac{\mu_x}{\sqrt{M_{i,j}^{n+1}}}\left[\sqrt{M_{i,j}^{n+1}M_{i+1,j}^{n+1}}\left(\frac{h_{i+1,j}^{*}}{\sqrt{M_{i+1,j}^{n+1}}}-\frac{h_{i,j}^{*}}{\sqrt{M_{i,j}^{n+1}}}\right) + \sqrt{M_{i-1,j}^{n+1}M_{i,j}^{n+1}}\left(\frac{h_{i-1,j}^{*}}{\sqrt{M_{i-1,j}^{n+1}}}-\frac{h_{i,j}^{*}}{\sqrt{M_{i,j}^{n+1}}}\right)\right]\nonumber\\
    &+\frac{\rho_{i,j}^{n}}{\sqrt{M_{i.j}^{n+1}}}.
\end{align}
Let $(k,l)$ be the indices satisfy $h_{k,l}^{*}/\sqrt{M_{k,l}^{n+1}}={\rm min}_{i,j}\{h_{i,j}^{*}/\sqrt{M_{i,j}^{n+1}}\}$. Note that $M_{i,j}^{n+1}\geq 0$ and $\rho_{i,j}^{n}\geq 0$ for arbitrary $i,j$, then we conclude $h_{k,l}^{*}/\sqrt{M_{k,l}^{n+1}}\geq0$ from \eqref{eqn:proofPositivity} because both terms in the bracket of RHS of \eqref{eqn:proofPositivity} are nonnegative. Thus, ${\rm min}_{i,j}\{h_{i,j}^{*}/\sqrt{M_{i,j}^{n+1}}\}\geq 0$ which implies $h_{i,j}^{*}\geq0$ for all $i,j$'s. The second equation of the scheme for the density, Eq. \eqref{eqn:semi-discrete4ADI_rho12} could be verified similarly. It means that we have shown the positivity preservation of our ADI scheme for the density $\rho$.
\end{proof}
\end{theorem}
\begin{remark}
Expanding the ADI scheme \eqref{eqn:semi-discrete4ADI_c3} and \eqref{eqn:semi-discrete4ADI_c4} for the chemoattractant concentration $c$
\begin{align}
    c_{i,j}^{*}&=\mu_{x}^{\varepsilon}(c_{i+1,j}^{*}-2c_{i,j}^{*}+c_{i-1,j}^{*})+c_{i,j}^{n}+\frac{\Delta t}{\varepsilon}\rho_{i,j}^{n},\\
    c_{i,j}^{n+1}&=\mu_{y}^{\varepsilon}(c_{i,j+1}^{n+1}-2c_{i,j}^{n+1}+c_{i,j-1}^{n+1})+c_{i,j}^{*},
\end{align}
we can show the positivity of concentration $c_{i,j}^{n}\geq 0$ following the same argument as in the proof of Theorem \ref{Thm1}.
\end{remark}

\subsubsection{Conservation of mass}
Mass conservation, $\frac{d}{dt}\int_{\Omega}\rho dx=0$, is another important property of the Keller-Segel equations. Therefore, we hope to preserve this property in the spatial discretization. In order to verify mass conservation of the proposed ADI scheme \eqref{eqn:semi-discrete4ADI_rho2}, we only need to verify $\sum_{i,j}\sqrt{M_{i,j}^{n+1}}\tau_x h_{i,j}^{n+1}=0$, $\sum_{i,j}\sqrt{M_{i,j}^{n+1}}\tau_y h_{i,j}^{n+1}=0$ and $\sum_{i,j}\sqrt{M_{i,j}^{n+1}}\tau_x\tau_y h_{i,j}^{n+1}=0$ separately. For example, we show that the fully discrete scheme of operator $\tau_x$ \eqref{eqn:ADI_rho_x} satisfies $\sum_{i,j}\sqrt{M_{i,j}^{n+1}}\tau_x h_{i,j}^{n+1}=0$. To be specific, taking $M_{i+\frac{1}{2},j}^{n+1}$ to be the geometric mean $\sqrt{M_{i,j}^{n+1}M_{i+1,j}^{n+1}}$, we have
\begin{align}
    \tau_x h_{i,j}^{n+1}&=\frac{1}{\sqrt{M_{i,j}^{n+1}}}\left[\sqrt{M_{i+1,j}^{n+1}M_{i,j}^{n+1}}\left(\frac{h_{i+1,j}^{n+1}}{\sqrt{M_{i+1,j}^{n+1}}}-\frac{h_{i,j}^{n+1}}{\sqrt{M_{i,j}^{n+1}}}\right)-\sqrt{{M_{i,j}^{n+1}M_{i-1,j}^{n+1}}}\left(\frac{h_{i,j}^{n+1}}{\sqrt{M_{i,j}^{n+1}}}-\frac{h_{i-1,j}^{n+1}}{\sqrt{M_{i-1,j}^{n+1}}}\right)\right]\nonumber\\
    &=\frac{1}{\sqrt{M_{i,j}^{n+1}}}\left[\left(\sqrt{M_{i,j}^{n+1}}h_{i+1,j}^{n+1}-\sqrt{M_{i-1,j}^{n+1}}h_{i,j}^{n+1}\right) + \left(\sqrt{M_{i,j}^{n+1}}h_{i-1,j}^{n+1}-\sqrt{M_{i+1,j}^{n+1}}h_{i,j}^{n+1}\right)\right].
\end{align}
Summing over $(i,j)$ of the above equation and assuming periodic boundary condition or other mass-conservative boundary conditions, we get
\begin{align}\label{eqn:mass-tau_x}
    \sum_{i,j}\sqrt{M_{i,j}^{n+1}}\tau_x h_{i,j}^{n+1}=0.
\end{align}
Similarly, we can verify that the fully discrete scheme of operators $\tau_y$ \eqref{eqn:ADI_rho_y} and $\tau_x\tau_y$ \eqref{eqn:tauxy} have mass conservation, i.e.,
\begin{align} \label{eqn:mass-tau_xy} 
     \sum_{i,j}\sqrt{M_{i,j}^{n+1}}\tau_y h_{i,j}^{n+1}=0,\quad
    \sum_{i,j}\sqrt{M_{i,j}^{n+1}}\tau_x\tau_y h_{i,j}^{n+1}=0.
\end{align}
Combining them into the scheme of density equation \eqref{eqn:semi-discrete4ADI_rho2} and recalling $\rho_{i,j}^{n+1}=\sqrt{M_{i,j}^{n+1}}h_{i,j}^{n+1}$, one can verify that the mass conservation of density variable $\rho$, that is, $\sum_{i,j}\rho_{i,j}^{n+1}=\sum_{i,j}\rho_{i,j}^{n}$. We note that the mass conservation of the concentration is $\sum_{i,j}c_{i,j}^{n+1}=\sum_{i,j}c_{i,j}^{n}+\frac{\Delta t}{\varepsilon}\sum_{i,j}\rho_{i,j}^{n}$ rather than $\sum_{i,j}c_{i,j}^{n+1}=\sum_{i,j}c_{i,j}^{n}$, which can be shown in the same manner.
\begin{remark}
    Since $0\leq i\leq N_x$ and $0\leq j\leq N_y$, we remark that the summation $\sum_{i,j}$ should be replaced by $\sum_{i,j=0}^{N_x-1,N_y-1}$ for periodic boundary conditions or $\sum_{i,j=1}^{N_x-1,N_y-1}$ for Neumann-type boundary conditions.
\end{remark}

\subsubsection{Energy dissipation}
It is well known that the  free energy of the Keller-Segel system satisfies an entropy-dissipation law. To be specific, the free energy of the parabolic-parabolic Keller-Segel system \eqref{eqn:original_rho} and \eqref{eqn:original_c} can be expressed as
\begin{align}\label{eqn:original_energy}
    E(\rho,c)=\int_{\Omega}\left[ \rho\log\rho-\rho-\rho c+\frac{1}{2}|\nabla c|^2\right]d\bold{x},
\end{align}
and the entropy-dissipation equality
\begin{align}\label{eqn:entropy-dissipation}
    \frac{d}{dt}E(t)=-\int_{\Omega}\left[ \rho|\nabla(\log\rho-c)|^2+\varepsilon|\partial_t c|^2\right]d\bold{x},
\end{align}
where the domain $\Omega\subset\mathbb{R}^2$ is bounded.\\
From the framework of energetic variational approach \cite{giga2017variational}, we can define the flow map $\mathbf{x}(\mathbf{X},t)$ for the cell density $\rho$ based on the mass conservation law: 
\begin{equation}
\begin{cases}
    \mathbf{u}(\mathbf{x},t) = \mathbf{x}_{t} ,\ t > 0,  \\
    \mathbf{x}(\mathbf{X},0) =  \mathbf{X}, 
\end{cases}
\end{equation} with $\mathbf{x} $ being the Eulerian coordinates, and $\mathbf{X}$ Lagrangian coordinates. Then the mass conservation law can be written as 
\begin{equation}
    \rho_{t} + \nabla \cdot (\rho \mathbf{u}) = 0,  
\end{equation} which implies 
\begin{equation}
    \rho(\mathbf{x}(\mathbf{X},t),t) = \frac{\rho_{0}(\mathbf{X})}{{\rm det}\frac{\partial \mathbf{x}}{\partial \mathbf{X}}}.
\end{equation}
With the flow map on $\rho$, the energy dissipation system \eqref{eqn:entropy-dissipation} becomes
\begin{equation}
    \label{eqn:energy_dissipation}
    \frac{d}{dt} E(\rho(\mathbf{x}(\mathbf{X},t),t),c) = -\int_{\Omega}[\rho|\mathbf{u}|^{2} + \epsilon|\partial_{t}c|^{2}]d\mathbf{x} = - 2 \mathcal{D},
\end{equation}where $c$ can be viewed as a fixed given function on the flow map of $\rho$ and the energy dissipation $\mathcal{D}(\mathbf{u},\partial_t c):=\frac{1}{2}\int_{\Omega}[\rho|\mathbf{u}|^{2} + \epsilon|\partial_{t}c|^{2}]d\mathbf{x}$. Now, we may define the action functional $\mathcal{A}(\mathbf{x}) = \int_{0}^{T} -E dt$ with respect to $\mathbf{x}$. According to the force balance law
\begin{equation}
    \label{eqn:force_balance}
    \frac{\delta \mathcal{A}}{\delta \mathbf{x}} = \frac{\delta \mathcal{D}}{\delta \mathbf{x}_{t}},
\end{equation}
we can get the first equation \eqref{eqn:original_rho} in the Keller-Segel system.
On the other hand, we may define another flow map $\mathbf{x}_{1}(\mathbf{X},t)$ for the chemoattractant concentration $c$ which satisfies the transport equation \cite{liu2020variational}: 
\begin{equation} 
    \label{eqn:transport_equation}
    \partial_{t} c + \mathbf{u}_{1} \cdot \nabla c = 0,
\end{equation} 
where \eqref{eqn:transport_equation} is equivalent to $c(\mathbf{x_{1}}(\mathbf{X},t),t) = c_{0}(\mathbf{X})$. For this flow map, the energy-dissipation system can be seen as 
\begin{equation}
    \frac{d}{dt} E(\rho(\mathbf{x_{1}},t),c_0(\mathbf{X}),F^{-1}\nabla_{\mathbf{X}} c_0(\mathbf{X})) = - 2 \Tilde{\mathcal{D}},
\end{equation}
where the Jacobian matrix $F = \frac{\partial \mathbf{x_1}}{\partial \mathbf{X}}$ and $\Tilde{\mathcal{D}}(\mathbf{u},\mathbf{u_1}):=\int_{\Omega} \rho |\mathbf{u}|^{2} + \epsilon |\mathbf{u_1}\cdot\nabla c|^{2} d\mathbf{x} $.\\
Now by the force balance law \eqref{eqn:force_balance} with $\mathbf{x}$ replaced by $\mathbf{x_{1}}$, we may obtain the second equation \eqref{eqn:original_c} of the Keller-Segel system. We refer the reader to \cite{forster2013mathematical,giga2017variational} for more background knowledge of the energetic variational approach.\\
In this part, we first prove that the original scheme \eqref{eqn:semidiscrete_c} and \eqref{eqn:semidiscrete_rho} for Keller-Segel equations proposed by Liu {\it{et al}}. \cite{Finitedifference} is energy dissipative. Subsequently, we discuss the energy dissipation law for the ADI scheme \eqref{eqn:semi-discrete4ADI_c2} and \eqref{eqn:semi-discrete4ADI_rho2}.\\
To this end, we impose two alternative boundary conditions on the density $\rho$ and the concentration $c$. One is the periodic boundary conditions for $\rho$ and $c$. The other is the following Neumann-type boundary conditions
\begin{align}\label{bd_cond}
     M\frac{\partial\left(\frac{\rho}{M}\right)}{\partial n}=0 \quad\text{and}\quad \frac{\partial c}{\partial n}=0, \quad\text{on}\quad \partial\Omega.
\end{align}
We first summarize the semi-discrete scheme result as the following theorem.
\begin{theorem}
    For the solution to the semi-discrete scheme \eqref{eqn:semi-symmtric_form_rho} and \eqref{eqn:semi-symmtric_form_c}, we have the following energy dissipation inequality:
    \begin{align}\label{eqn:semi_analytic_energylaw}
        E^{n+1}-E^{n}\leq -\Delta t\int_{\Omega}\left[\rho^{n+1}|\nabla(\log\rho^{n+1}-c^{n+1})|^2+\varepsilon\left|\frac{c^{n+1}-c^{n}}{\Delta t}\right|^2 \right]d\bold{x},\quad\text{for all}\quad n\geq 0,
    \end{align}
    where 
    \begin{align}\label{eqn:semi_energy}
        E^{n}=E(\rho^n,c^n)=\int_{\Omega}\left[ \rho^{n}\log\rho^{n}-\rho^{n}-\rho^{n} c^{n}+\frac{1}{2}|\nabla c^{n}|^2\right]d\bold{x}.
    \end{align}
\end{theorem}
\begin{proof}
Using the equality
\begin{align}\label{eqn:inequality_1}
        (a-b)\log a=(a\log a-a)-(b\log b-b)+\frac{(a-b)^2}{2\psi}, \quad\psi\in\left[\min\{a,b\},\max\{a,b\}\right],
    \end{align}
we have
\begin{align}\label{eqn:energy_proof_1}
    E^{n+1}-E^n=&\int_{\Omega}\left[\rho^{n+1}(\log\rho^{n+1}-1)-\rho^{n}(\log\rho^{n}-1)\right]-\rho^{n+1}c^{n+1}+\rho^{n}c^{n}+\frac{1}{2}(|\nabla c^{n+1}|^2-|\nabla c^{n}|^2)d\bold{x}\nonumber\\
    \leq&\int_{\Omega}\left[(\rho^{n+1}-\rho^{n})\log\rho^{n+1}-\rho^{n+1}c^{n+1}+\rho^{n}c^{n}+\frac{1}{2}(|\nabla c^{n+1}|^2-|\nabla c^{n}|^2)\right]d\bold{x}\nonumber\\
    =&\int_{\Omega}(\rho^{n+1}-\rho^{n})(\log\rho^{n+1}-c^{n+1})d\bold{x} + \int_{\Omega}\frac{1}{2}(|\nabla c^{n+1}|^2-|\nabla c^{n}|^2)d\bold{x} - \int_{\Omega}\rho^n(c^{n+1}-c^n)d\bold{x}.
\end{align}
Using Eq. \eqref{eqn:semi-symmtric_form_rho}, integration by parts and either of the boundary conditions on the variable $\rho$, from \eqref{eqn:energy_proof_1}, we obtain
\begin{align}\label{eqn:energy_proof_2}
    E^{n+1}-E^n\leq&\Delta t\int_{\Omega}\nabla\cdot\left[M^{n+1}\nabla\left(\frac{\rho^{n+1}}{M^{n+1}}\right)\right](\log\rho^{n+1}-c^{n+1})d\bold{x} + \int_{\Omega}\frac{1}{2}(|\nabla c^{n+1}|^2-|\nabla c^{n}|^2)d\bold{x}\nonumber\\
    &- \int_{\Omega}\rho^n(c^{n+1}-c^n)d\bold{x}\nonumber\\
    =&-\Delta t\int_{\Omega} M^{n+1}\nabla\left(\frac{\rho^{n+1}}{M^{n+1}}\right)\cdot\nabla(\log\rho^{n+1}-c^{n+1})d\bold{x} + \int_{\Omega}\frac{1}{2}(|\nabla c^{n+1}|^2-|\nabla c^{n}|^2)d\bold{x}\nonumber\\
    &- \int_{\Omega}\rho^n(c^{n+1}-c^n)d\bold{x}.
\end{align}
For two arbitrary scalars or vectors $\alpha,\beta$, using the equality
\begin{align}\label{eqn:inequality_2}
        (\alpha-\beta)\cdot\alpha=\frac{1}{2}(|\alpha|^2-|\beta|^2+|\alpha-\beta|^2),
\end{align}
integration by parts and either of the boundary conditions on the variable $c$, from \eqref{eqn:energy_proof_2}, we have
\begin{align}\label{eqn:energy_proof_3}
    E^{n+1}-E^n\leq& -\Delta t\int_{\Omega} M^{n+1}\nabla\left(\frac{\rho^{n+1}}{M^{n+1}}\right)\cdot\nabla(\log\rho^{n+1}-c^{n+1})d\bold{x} + \int_{\Omega}(\nabla c^{n+1}-\nabla c^{n})\cdot\nabla c^{n+1}d\bold{x}\nonumber\\
    &- \int_{\Omega}\rho^n(c^{n+1}-c^n)d\bold{x}\nonumber\\
    =&-\Delta t\int_{\Omega} M^{n+1}\nabla\left(\frac{\rho^{n+1}}{M^{n+1}}\right)\cdot\nabla(\log\rho^{n+1}-c^{n+1})d\bold{x} - \int_{\Omega}(c^{n+1}-c^{n})\Delta c^{n+1}d\bold{x}\nonumber\\
    &- \int_{\Omega}\rho^n(c^{n+1}-c^n)d\bold{x}\nonumber\\
    =&-\Delta t\int_{\Omega} M^{n+1}\nabla\left(\frac{\rho^{n+1}}{M^{n+1}}\right)\cdot\nabla(\log\rho^{n+1}-c^{n+1})d\bold{x} - \int_{\Omega}(c^{n+1}-c^{n})(\Delta c^{n+1}+\rho^n)d\bold{x}.
\end{align}
Using Eq. \eqref{eqn:semi-symmtric_form_c}, from \eqref{eqn:energy_proof_3}, we have
\begin{align}\label{eqn:energy_proof_4}
    E^{n+1}-E^n\leq&-\Delta t\int_{\Omega} M^{n+1}\nabla\left(\frac{\rho^{n+1}}{M^{n+1}}\right)\cdot\nabla(\log\rho^{n+1}-c^{n+1})d\bold{x} - \varepsilon\Delta t\int_{\Omega}\left|\frac{c^{n+1}-c^{n}}{\Delta t}\right|^2 d\bold{x}\nonumber\\
    =&-\Delta t\int_{\Omega}\left[\rho^{n+1}\left|\nabla(\log\rho^{n+1}-c^{n+1})\right|^2+\varepsilon\left|\frac{c^{n+1}-c^{n}}{\Delta t}\right|^2\right] d\bold{x}.
\end{align}
    For the last equality in \eqref{eqn:energy_proof_4}, the identical equation $M\nabla(\frac{\rho}{M})=\rho\nabla(\log\rho-c)$ are employed. We have shown the following energy dissipation law for the semi-discrete scheme \eqref{eqn:semi-symmtric_form_c} and \eqref{eqn:semi-symmtric_form_rho}: For all $n\geq 0$,
    \begin{align}
        E^{n+1}-E^{n}\leq -\Delta t\int_{\Omega}\left[\rho^{n+1}\left|\nabla(\log\rho^{n+1}-c^{n+1})\right|^2+\varepsilon\left|\frac{c^{n+1}-c^{n}}{\Delta t}\right|^2 \right]d\bold{x}.
    \end{align}
\end{proof}
To prove that the fully discrete five-point finite difference scheme \eqref{eqn:semidiscrete_c} and \eqref{eqn:semidiscrete_rho} preserves a discrete version of the energy dissipation, we state two propositions first. For convenience,
we assume the number of grid points $N=N_x=N_y$ and $0\leq i,j\leq N$.\\
We introduce several different inner products for two scalar functions $f(x,y)$ and $g(x,y)$ as follows
\begin{align}
    \label{innproduct_k}\left<f,g\right>_{k}&:=\sum_{i,j=1}^{N-1}f_{i,j}g_{i,j}\Delta x\Delta y,\\
    \label{innproduct_mx}\left<f,g\right>_{m_x}&:=\sum_{i=0,j=1}^{N-1}f_{i+\frac{1}{2},j}g_{i+\frac{1}{2},j}\Delta x\Delta y,\\
    \label{innproduct_my}\left<f,g\right>_{m_y}&:=\sum_{i=1,j=0}^{N-1}f_{i,j+\frac{1}{2}}g_{i,j+\frac{1}{2}}\Delta x\Delta y,
\end{align}
where $f_{i+\frac{1}{2},j}$ denotes the value at $(x_i+\frac{1}{2}\Delta x,y_j)$, etc.\\
For 2D vector-valued functions $\mathbf{u}(x,y)=(u_1(x,y),u_2(x,y))$ and $\mathbf{v}(x,y)=(v_1(x,y),v_2(x,y))$, we define the corresponding inner products as
\begin{align}
    \left<\mathbf{u},\mathbf{v}\right>_k&:=\left<u_1,v_1\right>_{k}+\left<u_2,v_2\right>_{k},\\
    \left<\mathbf{u},\mathbf{v}\right>_{m}&:=\left<u_1,v_1\right>_{m_x}+\left<u_2,v_2\right>_{m_y}.
\end{align}
Besides, we introduce the discrete gradient operator $\bm{\delta}=\left(\frac{1}{\Delta x}\delta_x,\frac{1}{\Delta y}\delta_y\right).$\\
Similar to Flavell {\it{et al}}. \cite{Li_2017}, we have the following two summation-by-parts formulas.
\begin{proposition}\label{prop_sumbypart1}
    The following summation-by-parts formula holds
    \begin{align}\label{sumbypart1}
        &\left<\frac{\sqrt{M^{n+1}}}{(\Delta x)^2}\tau_x\left(\frac{\rho^{n+1}}{\sqrt{M^{n+1}}}\right)+\frac{\sqrt{M^{n+1}}}{(\Delta y)^2}\tau_y\left(\frac{\rho^{n+1}}{\sqrt{M^{n+1}}}\right),\log\rho^{n+1}-c^{n+1}\right>_k\nonumber\\
        =&-\left<M^{n+1}\bm{\delta}\left(\frac{\rho^{n+1}}{M^{n+1}}\right),\bm{\delta}(\log\rho^{n+1}-c^{n+1})\right>_m\nonumber\\
        &-\frac{\Delta y}{\Delta x}\sum_{j=1}^{N-1}\left[M_{\frac{1}{2},j}^{n+1}\delta_x\left(\frac{\rho}{M}\right)_{\frac{1}{2},j}^{n+1}\left(\log\rho_{0,j}^{n+1}-c_{0,j}^{n+1}\right)\right]\nonumber\\
        &+\frac{\Delta y}{\Delta x}\sum_{j=1}^{N-1}\left[M_{N-\frac{1}{2},j}^{n+1}\delta_x\left(\frac{\rho}{M}\right)_{N-\frac{1}{2},j}^{n+1}\left(\log\rho_{N,j}^{n+1}-c_{N,j}^{n+1}\right)\right]\nonumber\\
        &-\frac{\Delta x}{\Delta y}\sum_{i=1}^{N-1}\left[M_{i,\frac{1}{2}}^{n+1}\delta_y\left(\frac{\rho}{M}\right)_{i,\frac{1}{2}}^{n+1}\left(\log\rho_{i,0}^{n+1}-c_{i,0}^{n+1}\right)\right]\nonumber\\
        &+\frac{\Delta x}{\Delta y}\sum_{i=1}^{N-1}\left[M_{i,N-\frac{1}{2}}^{n+1}\delta_y\left(\frac{\rho}{M}\right)_{i,N-\frac{1}{2}}^{n+1}\left(\log\rho_{i,N}^{n+1}-c_{i,N}^{n+1}\right)\right].
    \end{align}
    Furthermore, implementing the Neumann-type boundary conditions \eqref{bd_cond} as $\delta_x\left(\frac{\rho}{M}\right)_{\frac{1}{2},j}^{n+1}=\delta_x\left(\frac{\rho}{M}\right)_{N-\frac{1}{2},j}^{n+1}=\delta_y\left(\frac{\rho}{M}\right)_{i,\frac{1}{2}}^{n+1}=\delta_y\left(\frac{\rho}{M}\right)_{i,N-\frac{1}{2}}^{n+1}=0$, we obtain the following simplified version of summation-by-parts formula
    \begin{align}\label{sumbypart1_noBoundary}
        &\left<\frac{\sqrt{M^{n+1}}}{(\Delta x)^2}\tau_x\left(\frac{\rho^{n+1}}{\sqrt{M^{n+1}}}\right)+\frac{\sqrt{M^{n+1}}}{(\Delta y)^2}\tau_y\left(\frac{\rho^{n+1}}{\sqrt{M^{n+1}}}\right),\log\rho^{n+1}-c^{n+1}\right>_k\nonumber\\
        &=-\left<M^{n+1}\bm{\delta}\left(\frac{\rho^{n+1}}{M^{n+1}}\right),\bm{\delta}\left(\log\rho^{n+1}-c^{n+1}\right)\right>_m.
    \end{align}
\end{proposition}
\begin{proof}
    By the definition, we have
    \begin{align}\label{eqn:prop1_proof_1}
        &\left<\frac{\sqrt{M^{n+1}}}{(\Delta x)^2}\tau_x\left(\frac{\rho^{n+1}}{\sqrt{M^{n+1}}}\right)+\frac{\sqrt{M^{n+1}}}{(\Delta y)^2}\tau_y\left(\frac{\rho^{n+1}}{\sqrt{M^{n+1}}}\right),\log\rho^{n+1}-c^{n+1}\right>_k\nonumber\\
        =&\Delta x\Delta y\sum_{i,j=1}^{N-1}\left\{\frac{1}{(\Delta x)^2}\delta_x\left[M^{n+1}\delta_x\left(\frac{\rho^{n+1}}{M^{n+1}}\right)\right]_{i,j}+\frac{1}{(\Delta y)^2}\delta_y\left[M^{n+1}\delta_y\left(\frac{\rho^{n+1}}{M^{n+1}}\right)\right]_{i,j}\right\}\left(\log\rho_{i,j}^{n+1}-c_{i,j}^{n+1}\right)\nonumber\\
        =&\Delta y\sum_{i,j=1}^{N-1}\left[\frac{M_{i+\frac{1}{2},j}^{n+1}\delta_x\left(\frac{\rho}{M}\right)_{i+\frac{1}{2},j}^{n+1}-M_{i-\frac{1}{2},j}^{n+1}\delta_x\left(\frac{\rho}{M}\right)_{i-\frac{1}{2},j}^{n+1}}{\Delta x}\right]\left(\log\rho_{i,j}^{n+1}-c_{i,j}^{n+1}\right)\nonumber\\
        &+\Delta x\sum_{i,j=1}^{N-1}\left[\frac{M_{i,j+\frac{1}{2}}^{n+1}\delta_y\left(\frac{\rho}{M}\right)_{i,j+\frac{1}{2}}^{n+1}-M_{i,j-\frac{1}{2}}^{n+1}\delta_y\left(\frac{\rho}{M}\right)_{i,j-\frac{1}{2}}^{n+1}}{\Delta y}\right]\left(\log\rho_{i,j}^{n+1}-c_{i,j}^{n+1}\right).
    \end{align}
    For the first term on the right-hand side of \eqref{eqn:prop1_proof_1}, we can split it into two parts as follows
    \begin{align}\label{eqn:prop1_proof_2}
        &\Delta y\sum_{i,j=1}^{N-1}\left[\frac{M_{i+\frac{1}{2},j}^{n+1}\delta_x(\frac{\rho}{M})_{i+\frac{1}{2},j}^{n+1}-M_{i-\frac{1}{2},j}^{n+1}\delta_x\left(\frac{\rho}{M}\right)_{i-\frac{1}{2},j}^{n+1}}{\Delta x}\right]\left(\log\rho_{i,j}^{n+1}-c_{i,j}^{n+1}\right)\nonumber\\
        =&\frac{\Delta y}{\Delta x}\sum_{i,j=1}^{N-1}\left[M_{i+\frac{1}{2},j}^{n+1}\delta_x\left(\frac{\rho}{M}\right)_{i+\frac{1}{2},j}^{n+1}\left(\log\rho_{i,j}^{n+1}-c_{i,j}^{n+1}\right)\right]\nonumber\\
        &-\frac{\Delta y}{\Delta x}\sum_{i,j=1}^{N-1}\left[M_{i-\frac{1}{2},j}^{n+1}\delta_x\left(\frac{\rho}{M}\right)_{i-\frac{1}{2},j}^{n+1}\left(\log\rho_{i,j}^{n+1}-c_{i,j}^{n+1}\right)\right]\nonumber\\
        =&\frac{\Delta y}{\Delta x}\sum_{i,j=1}^{N-1}\left[M_{i+\frac{1}{2},j}^{n+1}\delta_x\left(\frac{\rho}{M}\right)_{i+\frac{1}{2},j}^{n+1}\left(\log\rho_{i,j}^{n+1}-c_{i,j}^{n+1}\right)\right]\nonumber\\
        &-\frac{\Delta y}{\Delta x}\sum_{i=0}^{N-2}\sum_{j=1}^{N-1}\left[M_{i+\frac{1}{2},j}^{n+1}\delta_x\left(\frac{\rho}{M}\right)_{i+\frac{1}{2},j}^{n+1}\left(\log\rho_{i+1,j}^{n+1}-c_{i+1,j}^{n+1}\right)\right]\nonumber\\
        =&-\Delta x\Delta y\sum_{i=0,j=1}^{N-1}\left[M_{i+\frac{1}{2},j}^{n+1}\frac{\delta_x\left(\frac{\rho}{M}\right)_{i+\frac{1}{2},j}^{n+1}}{\Delta x}\frac{\left(\log\rho_{i+1,j}^{n+1}-c_{i+1,j}^{n+1}\right)-\left(\log\rho_{i,j}^{n+1}-c_{i,j}^{n+1}\right)}{\Delta x}\right]\nonumber\\
        &-\frac{\Delta y}{\Delta x}\sum_{j=1}^{N-1}\left[M_{\frac{1}{2},j}^{n+1}\delta_x\left(\frac{\rho}{M}\right)_{\frac{1}{2},j}^{n+1}\left(\log\rho_{0,j}^{n+1}-c_{0,j}^{n+1}\right)\right]\nonumber\\
        &+\frac{\Delta y}{\Delta x}\sum_{j=1}^{N-1}\left[M_{N-\frac{1}{2},j}^{n+1}\delta_x\left(\frac{\rho}{M}\right)_{N-\frac{1}{2},j}^{n+1}\left(\log\rho_{N,j}^{n+1}-c_{N,j}^{n+1}\right)\right]\nonumber\\
        =&-\frac{1}{(\Delta x)^2}\left<M^{n+1}\delta_x\left(\frac{\rho^{n+1}}{M^{n+1}}\right),\delta_x\left(\log\rho^{n+1}-c^{n+1}\right)\right>_{m_x}\nonumber\\
        &-\frac{\Delta y}{\Delta x}\sum_{j=1}^{N-1}\left[M_{\frac{1}{2},j}^{n+1}\delta_x\left(\frac{\rho}{M}\right)_{\frac{1}{2},j}^{n+1}\left(\log\rho_{0,j}^{n+1}-c_{0,j}^{n+1}\right)\right]\nonumber\\
        &+\frac{\Delta y}{\Delta x}\sum_{j=1}^{N-1}\left[M_{N-\frac{1}{2},j}^{n+1}\delta_x\left(\frac{\rho}{M}\right)_{N-\frac{1}{2},j}^{n+1}\left(\log\rho_{N,j}^{n+1}-c_{N,j}^{n+1}\right)\right].
    \end{align}
    Similarly, for the second term on the right-hand side of \eqref{eqn:prop1_proof_1}, we have
    \begin{align}\label{eqn:prop1_proof_3}
        &\Delta x\sum_{i,j=1}^{N-1}\left[\frac{M_{i,j+\frac{1}{2}}^{n+1}\delta_y\left(\frac{\rho}{M}\right)_{i,j\frac{1}{2}}^{n+1}-M_{i,j-\frac{1}{2}}^{n+1}\delta_y\left(\frac{\rho}{M}\right)_{i,j-\frac{1}{2}}^{n+1}}{\Delta y}\right]\left(\log\rho_{i,j}^{n+1}-c_{i,j}^{n+1}\right)\nonumber\\
        =&-\frac{1}{(\Delta y)^2}\left<M^{n+1}\delta_y\left(\frac{\rho^{n+1}}{M^{n+1}}\right),\delta_y\left(\log\rho^{n+1}-c^{n+1}\right)\right>_{m_y}\nonumber\\
        &-\frac{\Delta x}{\Delta y}\sum_{i=1}^{N-1}\left[M_{i,\frac{1}{2}}^{n+1}\delta_y\left(\frac{\rho}{M}\right)_{i,\frac{1}{2}}^{n+1}\left(\log\rho_{i,0}^{n+1}-c_{i,0}^{n+1}\right)\right]\nonumber\\
        &+\frac{\Delta x}{\Delta y}\sum_{i=1}^{N-1}\left[M_{i,N-\frac{1}{2}}^{n+1}\delta_y\left(\frac{\rho}{M}\right)_{i,N-\frac{1}{2}}^{n+1}\left(\log\rho_{i,N}^{n+1}-c_{i,N}^{n+1}\right)\right].
    \end{align}
    Combining \eqref{eqn:prop1_proof_2} and \eqref{eqn:prop1_proof_3}, we obtain the formula \eqref{sumbypart1}.\\
    If the Neumann-type boundary conditions \eqref{bd_cond} are discretized as $\delta_x\left(\frac{\rho}{M}\right)_{\frac{1}{2},j}^{n+1}=\delta_x\left(\frac{\rho}{M}\right)_{N-\frac{1}{2},j}^{n+1}=\delta_y\left(\frac{\rho}{M}\right)_{i,\frac{1}{2}}^{n+1}=\delta_y\left(\frac{\rho}{M}\right)_{i,N-\frac{1}{2}}^{n+1}=0$, \eqref{sumbypart1} is simplified to \eqref{sumbypart1_noBoundary}.
    
\end{proof}
\begin{remark}\label{remark_sumbypart_1}
    If the boundary conditions is given as periodic conditions, the simplified formula \eqref{sumbypart1_noBoundary} holds if we redefine the inner products in \eqref{innproduct_k}-\eqref{innproduct_my} as 
    \begin{align}
    \left<f,g\right>_{k}&:=\sum_{i,j=0}^{N-1}f_{i,j}g_{i,j}\Delta x\Delta y,\\
    \left<f,g\right>_{m_x}&:=\sum_{i,j=0}^{N-1}f_{i+\frac{1}{2},j}g_{i+\frac{1}{2},j}\Delta x\Delta y,\\
    \left<f,g\right>_{m_y}&:=\sum_{i,j=0}^{N-1}f_{i,j+\frac{1}{2}}g_{i,j+\frac{1}{2}}\Delta x\Delta y.
    \end{align}
\end{remark}
Similarly, we can also prove the following summation-by-parts formula.
\begin{proposition}\label{prop_sumbypart2}
    \begin{align}\label{sumbypart2}
    \left<c^{n+1}-c^n,\bm{\delta}^2c^{n+1}\right>_k=&-\left<\bm{\delta} c^{n+1},\bm{\delta} \left(c^{n+1}-c^n\right)\right>_{m}\nonumber\\
    &-\frac{\Delta y}{\Delta x}\sum_{j=1}^{N-1}(c_{0,j}^{n+1}-c_{0,j}^{n})\delta_xc_{\frac{1}{2},j}^{n+1}+\frac{\Delta y}{\Delta x}\sum_{j=1}^{N-1}(c_{N,j}^{n+1}-c_{N,j}^{n})\delta_xc_{N-\frac{1}{2},j}^{n+1}\nonumber\\
    &-\frac{\Delta x}{\Delta y}\sum_{i=1}^{N-1}(c_{i,0}^{n+1}-c_{i,0}^{n})\delta_xc_{i,\frac{1}{2}}^{n+1}+\frac{\Delta x}{\Delta y}\sum_{i=1}^{N-1}(c_{i,N}^{n+1}-c_{i,N}^{n})\delta_xc_{i,N-\frac{1}{2}}^{n+1}.
    \end{align}
    Furthermore, implementing the Neumann-type boundary conditions \eqref{bd_cond} as $\delta_x c_{\frac{1}{2},j}^{n+1}=\delta_x c_{N-\frac{1}{2},j}^{n+1}=\delta_y c_{i,\frac{1}{2}}^{n+1}=\delta_y c_{i,N-\frac{1}{2}}^{n+1}=0$, we have 
    \begin{align}\label{sumbypart2_noBoundary}
        \left<c^{n+1}-c^n,\bm{\delta}^2c^{n+1}\right>_k=&-\left<\bm{\delta} c^{n+1},\bm{\delta} \left(c^{n+1}-c^n\right)\right>_{m}.
    \end{align}
\end{proposition}
\begin{proof}
    The proof is similar to that of Proposition \ref{prop_sumbypart1}.
\end{proof}

\begin{remark}\label{remark_sumbypart_2}
    If the boundary conditions is given as periodic conditions, the simplified formula \eqref{sumbypart2_noBoundary} holds if we redefine the inner products in \eqref{innproduct_k}-\eqref{innproduct_my} as Remark \ref{remark_sumbypart_1}.
\end{remark}

Now we are ready to show the discrete version of the energy dissipation law. When the boundary conditions are of Neumann-type, we approximate $E^n$ by the fully discrete energy at time $t_n$ as 
\begin{equation}\label{eqn:fully_discrete_energy}
    \mathcal{E}^n=\Delta x\Delta y\sum_{i,j=1}^{N-1}\left[\rho_{i,j}^{n}\log\rho_{i,j}^{n}-\rho_{i,j}^{n}-\rho_{i,j}^{n}c_{i,j}^{n}\right]+\frac{1}{2}\Delta x\Delta y\left[\sum_{i=0,j=1}^{N-1}\left|\frac{\delta_x c_{i+\frac{1}{2},j}^{n}}{\Delta x}\right|^2+\sum_{i=1,j=0}^{N-1}\left|\frac{\delta_y c_{i,j+\frac{1}{2}}^{n}}{\Delta y}\right|^2\right].
\end{equation}
When the boundary conditions are periodic, we approximate $E^n$ \eqref{eqn:semi_energy} as 
\begin{align}
    \label{eqn:fully_discrete_energy_1}
    \mathcal{E}^n=\Delta x\Delta y\sum_{i,j=0}^{N-1}\left[\rho_{i,j}^{n}\log\rho_{i,j}^{n}-\rho_{i,j}^{n}-\rho_{i,j}^{n}c_{i,j}^{n}\right]+\frac{1}{2}\Delta x\Delta y\left[\sum_{i,j=0}^{N-1}\left|\frac{\delta_x c_{i+\frac{1}{2},j}^{n}}{\Delta x}\right|^2+\sum_{i,j=0}^{N-1}\left|\frac{\delta_y c_{i,j+\frac{1}{2}}^{n}}{\Delta y}\right|^2\right].
\end{align}
\begin{remark}
    Under Neumann-type boundary condtions, the discrete energy $\mathcal{E}^n$ \eqref{eqn:fully_discrete_energy} has the leading order quadratic error $O(\Delta x\Delta y)$ for approximating the analytic energy $E(\rho(t_n),c(t_n))$. Under the periodic boundary conditions, the first term of \eqref{eqn:fully_discrete_energy_1} approximates $\int_{\Omega}\left(\rho^n\log\rho^n-\rho^n-\rho^n c^n\right)d\bold{x}$ with spectral accuracy while the second term of \eqref{eqn:fully_discrete_energy_1} has the leading-order error of $O\left(\left(\Delta x\right)^3\Delta y\right)+O\left(\Delta x\left(\Delta y\right)^3\right)$ due to the second-order accuracy of central differencing $\frac{\delta_xc_{i+\frac{1}{2},j}^{n}}{\Delta x}$ to $\frac{\partial c}{\partial x}\left(x_{i+\frac{1}{2}},y_j,t_n\right)$ and $\frac{\delta_yc_{i,j+\frac{1}{2}}^{n}}{\Delta y}$ to $\frac{\partial c}{\partial y}\left(x_i,y_{j+\frac{1}{2}},t_n\right)$.
\end{remark}

The following discrete version of the energy dissipation law holds for the finite difference scheme \eqref{eqn:semidiscrete_c} and \eqref{eqn:semidiscrete_rho}.
\begin{theorem}\label{Thm:fullyEnergy}
    For the solution to the finite difference scheme \eqref{eqn:semidiscrete_c} and \eqref{eqn:semidiscrete_rho}, we have the following energy dissipation inequality:
    \begin{align}\label{eqn:fullyenergy}
        \mathcal{E}^{n+1}-\mathcal{E}^n\leq& -\Delta t\left<M^{n+1}\bm{\delta}\left(\frac{\rho^{n+1}}{M^{n+1}}\right),\bm{\delta}\left(\log\rho^{n+1}-c^{n+1}\right)\right>_m\nonumber\\
        &-\varepsilon\Delta t\left<\frac{c^{n+1}-c^{n}}{\Delta t},\frac{c^{n+1}-c^{n}}{\Delta t}\right>_k.
    \end{align}
\end{theorem}
\begin{proof}
    Using the equality \eqref{eqn:inequality_1}, we have
    \begin{align}\label{eqn:fullyDiscrete_proof_0}
         \mathcal{E}^{n+1}- \mathcal{E}^n=&\Delta x\Delta y\sum_{i,j=1}^{N-1}\left[\rho_{i,j}^{n+1}\left(\log\rho_{i,j}^{n+1}-1\right)-\rho_{i,j}^{n}\left(\log\rho_{i,j}^{n}-1\right)\right]\nonumber\\
        &-\Delta x\Delta y\sum_{i,j=1}^{N-1}\left[\rho_{i,j}^{n+1}c_{i,j}^{n+1}-\rho_{i,j}^{n}c_{i,j}^{n}\right]\nonumber\\
        &+\frac{1}{2}\frac{\Delta y}{\Delta x}\sum_{i=0,j=1}^{N-1}\left(\left|\delta_xc_{i+\frac{1}{2},j}^{n+1}\right|^2-\left|\delta_xc_{i+\frac{1}{2},j}^{n}\right|^2\right)+\frac{1}{2}\frac{\Delta x}{\Delta y}\sum_{i=1,j=0}^{N-1}\left(\left|\delta_yc_{i,j+\frac{1}{2}}^{n+1}\right|^2-\left|\delta_xc_{i,j+\frac{1}{2}}^{n}\right|^2\right)\nonumber\\
        \leq&\Delta x\Delta y\sum_{i,j=1}^{N-1}\left[\left(\rho_{i,j}^{n+1}-\rho_{i,j}^{n}\right)\log\rho_{i,j}^{n+1}\right]-\Delta x\Delta y\sum_{i,j=1}^{N-1}\left[\rho_{i,j}^{n+1}c_{i,j}^{n+1}-\rho_{i,j}^{n}c_{i,j}^{n}\right]\nonumber\\
        &+\frac{1}{2}\frac{\Delta y}{\Delta x}\sum_{i=0,j=1}^{N-1}\left(\left|\delta_xc_{i+\frac{1}{2},j}^{n+1}\right|^2-\left|\delta_xc_{i+\frac{1}{2},j}^{n}\right|^2\right)+\frac{1}{2}\frac{\Delta x}{\Delta y}\sum_{i=1,j=0}^{N-1}\left(\left|\delta_yc_{i,j+\frac{1}{2}}^{n+1}\right|^2-\left|\delta_xc_{i,j+\frac{1}{2}}^{n}\right|^2\right)\nonumber\\
        =&\Delta x\Delta y\sum_{i,j=1}^{N-1}\left[\left(\rho_{i,j}^{n+1}-\rho_{i,j}^{n}\right)\left(\log\rho_{i,j}^{n+1}-c_{i,j}^{n+1}\right)-\rho_{i,j}^{n}\left(c_{i,j}^{n+1}-c_{i,j}^{n}\right)\right]\nonumber\\
        &+\frac{1}{2}\frac{\Delta y}{\Delta x}\sum_{i=0,j=1}^{N-1}\left(\left|\delta_xc_{i+\frac{1}{2},j}^{n+1}\right|^2-\left|\delta_xc_{i+\frac{1}{2},j}^{n}\right|^2\right)+\frac{1}{2}\frac{\Delta x}{\Delta y}\sum_{i=1,j=0}^{N-1}\left(\left|\delta_yc_{i,j+\frac{1}{2}}^{n+1}\right|^2-\left|\delta_xc_{i,j+\frac{1}{2}}^{n}\right|^2\right).
    \end{align}
    Using the equality \eqref{eqn:inequality_2}, from \eqref{eqn:fullyDiscrete_proof_0}, we have
    \begin{align}\label{eqn:fullyDiscrete_proof_1}
         \mathcal{E}^{n+1}- \mathcal{E}^n\leq&\Delta x\Delta y\sum_{i,j=1}^{N-1}\left[\left(\rho_{i,j}^{n+1}-\rho_{i,j}^{n}\right)\left(\log\rho_{i,j}^{n+1}-c_{i,j}^{n+1}\right)-\rho_{i,j}^{n}\left(c_{i,j}^{n+1}-c_{i,j}^{n}\right)\right]\nonumber\\
        &+\frac{\Delta y}{\Delta x}\sum_{i=0,j=1}^{N-1}\left(\delta_xc_{i+\frac{1}{2},j}^{n+1}-\delta_xc_{i+\frac{1}{2},j}^{n}\right)\delta_xc_{i+\frac{1}{2},j}^{n+1}+\frac{\Delta x}{\Delta y}\sum_{i=1,j=0}^{N-1}\left(\delta_yc_{i,j+\frac{1}{2}}^{n+1}-\delta_yc_{i,j+\frac{1}{2}}^{n}\right)\delta_yc_{i,j+\frac{1}{2}}^{n+1}\nonumber\\
        =&\left<\rho^{n+1}-\rho^n,\log\rho^{n+1}-c^{n+1}\right>_k + \left<\frac{\delta_x c^{n+1}-\delta_x c^n}{\Delta x},\frac{\delta_x c^{n+1}}{\Delta x}\right>_{m_x}\nonumber\\
        &+ \left<\frac{\delta_y c^{n+1}-\delta_y c^n}{\Delta y},\frac{\delta_y c^{n+1}}{\Delta y}\right>_{m_y}-\left<c^{n+1}-c^n,\rho^n\right>_k. 
    \end{align}
    Using Eqs. \eqref{eqn:semidiscrete_c}-\eqref{eqn:semidiscrete_rho}, Proposition \ref{prop_sumbypart1}, Proposition \ref{prop_sumbypart2} and the Neumann boundary conditions, from \eqref{eqn:fullyDiscrete_proof_1}, we obtain
    \begin{align}\label{eqn:fullyDiscrete_proof_2}
         \mathcal{E}^{n+1}- \mathcal{E}^n\leq&\Delta t\left<\frac{\sqrt{M^{n+1}}}{(\Delta x)^2}\tau_x\left(\frac{\rho^{n+1}}{\sqrt{M^{n+1}}}\right)+\frac{\sqrt{M^{n+1}}}{(\Delta y)^2}\tau_y\left(\frac{\rho^{n+1}}{\sqrt{M^{n+1}}}\right),\log\rho^{n+1}-c^{n+1}\right>_k\nonumber\\
        &+ \left<\frac{\delta_x c^{n+1}-\delta_x c^n}{\Delta x},\frac{\delta_x c^{n+1}}{\Delta x}\right>_{m_x}+ \left<\frac{\delta_y c^{n+1}-\delta_y c^n}{\Delta y},\frac{\delta_y c^{n+1}}{\Delta y}\right>_{m_y}-\left<c^{n+1}-c^n,\rho^n\right>_k\nonumber\\
        =&\Delta t\left<\frac{\sqrt{M^{n+1}}}{(\Delta x)^2}\tau_x\left(\frac{\rho^{n+1}}{\sqrt{M^{n+1}}}\right)+\frac{\sqrt{M^{n+1}}}{(\Delta y)^2}\tau_y\left(\frac{\rho^{n+1}}{\sqrt{M^{n+1}}}\right),\log\rho^{n+1}-c^{n+1}\right>_k\nonumber\\
        &-\left<c^{n+1}-c^n,\bm{\delta}^2c^{n+1}+\rho^n\right>_k\nonumber\\
        =&-\Delta t\left[\left<M^{n+1}\bm{\delta}\left(\frac{\rho^{n+1}}{M^{n+1}}\right),\bm{\delta}\left(\log\rho^{n+1}-c^{n+1}\right)\right>_m+\varepsilon\left<\frac{c^{n+1}-c^n}{\Delta t},\frac{c^{n+1}-c^n}{\Delta t}\right>_k\right].
    \end{align}
    Therefore, we have completed the proof.
\end{proof}
\begin{remark}
    Now, we compare the discrete energy dissipation law \eqref{eqn:fullyenergy} with the semi-analytic one \eqref{eqn:semi_analytic_energylaw}.  For the fully discrete five-point scheme \eqref{eqn:semidiscrete_c} and \eqref{eqn:semidiscrete_rho}, we take 
    \begin{align}\label{eqn:fullyEnergy_proof1}
        \left(\frac{1}{\Delta x}M_{i+\frac{1}{2},j}^{n+1}\delta_x\left(\frac{\rho}{M}\right)_{i+\frac{1}{2},j}^{n+1},\frac{1}{\Delta y}M_{i,j+\frac{1}{2}}^{n+1}\delta_y\left(\frac{\rho}{M}\right)_{i,j+\frac{1}{2}}^{n+1}\right)
    \end{align}
    as an approximation for $M^{n+1}\nabla\left(\frac{\rho^{n+1}}{M^{n+1}}\right)$ with a second-order accuracy in space. Based on the identity $M^{n+1}\nabla\left(\frac{\rho}{M}\right)^{n+1}=\rho^{n+1}\nabla\left(\log\rho^{n+1}-c^{n+1}\right)$, Eq. \eqref{eqn:fullyEnergy_proof1} can be considered as 
    \begin{align}\label{eqn:fullyEnergy_proof2}
        \left(\frac{1}{\Delta x}\rho_{i+\frac{1}{2},j}^{n+1}\delta_x\left(\log\rho-c\right)_{i+\frac{1}{2},j}^{n+1},\frac{1}{\Delta y}\rho_{i,j+\frac{1}{2}}^{n+1}\delta_y\left(\log\rho-c\right)_{i,j+\frac{1}{2}}^{n+1}\right),
    \end{align}
    which approximates $\rho^{n+1}\nabla\left(\log\rho^{n+1}-c^{n+1}\right)$ with a second-order accuracy in space. Therefore, the discrete energy dissipation law (Theorem \ref{Thm:fullyEnergy}) gives
    
    \begin{align}
        \mathcal{E}^{n+1}-\mathcal{E}^n\leq& -\Delta t\left<M^{n+1}\bm{\delta}\left(\frac{\rho^{n+1}}{M^{n+1}}\right),\bm{\delta}\left(\log\rho^{n+1}-c^{n+1}\right)\right>_m-\varepsilon\Delta t\left<\frac{c^{n+1}-c^{n}}{\Delta t},\frac{c^{n+1}-c^{n}}{\Delta t}\right>_k\nonumber\\
        =&-\Delta t\int_{\Omega}\left[\rho^{n+1}|\nabla(\log\rho^{n+1}-c^{n+1})|^2+\varepsilon\left|\frac{c^{n+1}-c^{n}}{\Delta t}\right|^2 \right]d\bold{x}+O\left(\Delta t\left(\Delta x\right)^2\right)+O\left(\Delta t\left(\Delta y\right)^2\right).
    \end{align}
    In other words, we have shown the discrete version of the energy dissipation law for the fully discrete five-point scheme \eqref{eqn:semidiscrete_c} and \eqref{eqn:semidiscrete_rho} matches with the semi-analytic one \eqref{eqn:semi_analytic_energylaw} with second-order accuracy in space.
\end{remark}
\begin{remark}
    It can be shown that the energy law of the ADI schemes \eqref{eqn:semi-discrete4ADI_c2} and \eqref{eqn:semi-discrete4ADI_rho2} is 
    \begin{align}
        \mathcal{E}^{n+1}-\mathcal{E}^n\leq&-\Delta t\left[\left<\rho^{n+1}\bm{\delta}\left(\log\rho^{n+1}-c^{n+1}\right),\bm{\delta}\left(\log\rho^{n+1}-c^{n+1}\right)\right>_m+\varepsilon\left<\frac{c^{n+1}-c^n}{\Delta t},\frac{c^{n+1}-c^n}{\Delta t}\right>_k\right]\nonumber\\
        &-\Delta t\left< \mu_x\mu_y\sqrt{M^{n+1}}\tau_x\tau_y\left(\frac{\rho}{\sqrt{M}}\right)^{n+1},\log\rho^{n+1}-c^{n+1}\right>_k\nonumber\\
        &+\Delta t\left<c^{n+1}-c^n,\mu_{x}^{\varepsilon}\mu_{y}^{\varepsilon}\delta_{x}^{2}\delta_{y}^{2}c^{n+1}\right>_k+O\left(\Delta t\left(\Delta x\right)^2\right)+O\left(\Delta t\left(\Delta y\right)^2\right).
    \end{align}
    Notice that the high-order operators $\mu_x\mu_y\sqrt{M^{n+1}}\tau_x\tau_y\left(\frac{\rho}{\sqrt{M}}\right)^{n+1}$ and $\mu_{x}^{\varepsilon}\mu_{y}^{\varepsilon}\delta_{x}^{2}\delta_{y}^{2}c^{n+1}$ converge to zero with order $O((\Delta t)^2)$. Therefore, we know that the energy dissipation law of the ADI schemes \eqref{eqn:semi-discrete4ADI_c2} and \eqref{eqn:semi-discrete4ADI_rho2} holds when the grid sizes $\Delta x$, $\Delta y$ and the time step size $\Delta t$ tend to zero.
\end{remark}

\subsection{A second-order scheme}
Inspired by the Crank-Nicolson scheme, we can extend the previous ADI scheme to a second-order scheme in time. The preceding sections have demonstrated that our scheme exhibits second-order spatial accuracy. To achieve second-order accuracy in time, we propose an additive splitting ADI scheme that resembles a Crank-Nicolson scheme. To be specific, 
\begin{align}
    \label{eqn:additiveADI_c1}\varepsilon\frac{c^{n+\frac{1}{2}}-c^n}{\Delta t/2}&=\frac{1}{(\Delta x)^2}\delta_{x}^2c^{n+\frac{1}{2}}+\frac{1}{(\Delta y)^2}\delta_{y}^2c^{n}+\rho^{n},\\
    \label{eqn:additiveADI_c2}\varepsilon\frac{c^{n+1}-c^{n+\frac{1}{2}}}{\Delta t/2}&=\frac{1}{(\Delta x)^2}\delta_{x}^2c^{n+\frac{1}{2}}+\frac{1}{(\Delta y)^2}\delta_{y}^2c^{n+1}+\rho^{n+1},
\end{align}
and 
\begin{align}
    \label{eqn:additiveADI_rho1} \frac{\rho^{n+\frac{1}{2}}-\rho^n}{\Delta t/2}&=\sqrt{M^{n+\frac{1}{2}}}\left[\frac{1}{(\Delta x)^2}\bar{\tau}_{x}\left(\frac{\rho^{n+\frac{1}{2}}}{\sqrt{M^{n+\frac{1}{2}}}}\right)+\frac{1}{(\Delta y)^2}\bar{\tau}_{y}\left(\frac{\rho^{n}}{\sqrt{M^{n+\frac{1}{2}}}}\right)\right],\\
    \label{eqn:additiveADI_rho2} \frac{\rho^{n+1}-\rho^{n+\frac{1}{2}}}{\Delta t/2}&=\sqrt{M^{n+\frac{1}{2}}}\left[\frac{1}{(\Delta x)^2}\bar{\tau}_{x}\left(\frac{\rho^{n+\frac{1}{2}}}{\sqrt{M^{n+\frac{1}{2}}}}\right)+\frac{1}{(\Delta y)^2}\bar{\tau}_{y}\left(\frac{\rho^{n+1}}{\sqrt{M^{n+\frac{1}{2}}}}\right)\right].
\end{align}
Adding Eq. \eqref{eqn:additiveADI_c1} and Eq. \eqref{eqn:additiveADI_c2}, we have
\begin{align}
    \label{eqn:additiveADI_c} \varepsilon\frac{c^{n+1}-c^n}{\Delta t}=\left[ \frac{1}{(\Delta x)^2}\delta_{x}^{2}c^{n+\frac{1}{2}}+\frac{1}{2(\Delta y)^2}\delta_{y}^{2}(c^n+c^{n+1}) \right]+\frac{\rho^n+\rho^{n+1}}{2}.
\end{align}
Similarly, adding Eq. \eqref{eqn:additiveADI_rho1} and Eq. \eqref{eqn:additiveADI_rho2}, we have
\begin{align}
    \label{eqn:additiveADI_rho} \frac{\rho^{n+1}-\rho^n}{\Delta t}=\sqrt{M^{n+\frac{1}{2}}}\left[ \frac{1}{(\Delta x)^2}\bar{\tau}_{x}\left(\frac{\rho^{n+\frac{1}{2}}}{\sqrt{M^{n+\frac{1}{2}}}}\right)+\frac{1}{2(\Delta y)^2}\bar{\tau}_{y}\left(\frac{\rho^n}{\sqrt{M^{n+\frac{1}{2}}}}+\frac{\rho^{n+1}}{\sqrt{M^{n+\frac{1}{2}}}}\right) \right].
\end{align}
\begin{remark}
    The variable $\rho^{n+1}$ is unknown in \eqref{eqn:additiveADI_c2} or \eqref{eqn:additiveADI_c}. To preserve the second-order time convergence, we could use $2\rho^n-\rho^{n-1}$ to approximate $\rho^{n+1}$. Using this approximation, it is easy to show that the truncation errors of \eqref{eqn:additiveADI_c} and \eqref{eqn:additiveADI_rho} are $O(\Delta t)^2+O((\Delta x)^2)+O((\Delta y)^2)$. Moreover, this scheme is asymptotic preserving to the quasi-static limit. 
\end{remark}
The fully-discrete equation for the first-stage of the ADI scheme Eq. \eqref{eqn:additiveADI_c1} with regard to the concerntration $c$ is 
\begin{align}\label{eqn:additiveADI_fully_c}
    \varepsilon c_{i,j}^{n+\frac{1}{2}}&=\frac{\mu_x}{2}\left(c_{i-1,j}^{n+\frac{1}{2}}-2c_{i,j}^{n+\frac{1}{2}}+c_{i+1,j}^{n+\frac{1}{2}}\right)+\frac{\mu_y}{2}\left(c_{i,j-1}^{n}-2c_{i,j}^{n}+c_{i,j+1}^{n}\right)+\varepsilon c_{i,j}^{n}+\frac{\Delta t}{2}\rho_{i,j}^{n},\nonumber\\
    &=\frac{\mu_x}{2}\left(c_{i-1,j}^{n+\frac{1}{2}}-2c_{i,j}^{n+\frac{1}{2}}+c_{i+1,j}^{n+\frac{1}{2}}\right)+\left[ \frac{\mu_y}{2}c_{i,j-1}^{n}+(\varepsilon-\mu_y)c_{i,j}^{n}+\frac{\mu_y}{2}c_{i,j+1}^{n}\right]+\frac{\Delta t}{2}\rho_{i,j}^{n},
\end{align}
where $\mu_x=\frac{\Delta t}{(\Delta x)^2}, \mu_y=\frac{\Delta t}{(\Delta y)^2}$ as in Eq. \eqref{eqn:semi-discrete4ADI_rho1}. If $\varepsilon\geq\mu_y$, the second bracket on the RHS is nonnegative. Subsequently, we can prove that \eqref{eqn:additiveADI_c1} preserves positivity in the same way as we did in the first-order accuracy in time ADI scheme \eqref{eqn:semi-discrete4ADI_c2}-\eqref{eqn:semi-discrete4ADI_c4}. Similarly, if $\varepsilon\geq \mu_x$, then \eqref{eqn:additiveADI_c2} preserves positivity. That is, our scheme \eqref{eqn:additiveADI_c} for the concentration $c$ preserves positivity, if $\varepsilon\geq\max(\mu_x,\mu_y)$.\\
The fully discrete equation for the first-stage of the ADI scheme Eq. \eqref{eqn:additiveADI_rho1} for the density $\rho$ is 
\begin{align}\label{eqn:2nd-scheme1}
    \rho_{i,j}^{n+\frac{1}{2}}&=\frac{\mu_x\sqrt{M_{i,j}^{n+\frac{1}{2}}}}{2}\left[\left(\frac{\rho_{i-1,j}^{n+\frac{1}{2}}}{\sqrt{M_{i-1,j}^{n+\frac{1}{2}}}}\right)+\left(\frac{\rho_{i+1,j}^{n+\frac{1}{2}}}{\sqrt{M_{i+1,j}^{n+\frac{1}{2}}}}\right)-\left( \frac{\sqrt{M_{i-1,j}^{n+\frac{1}{2}}}+\sqrt{M_{i+1,j}^{n+\frac{1}{2}}}}{M_{i,j}^{n+\frac{1}{2}}}\right)\rho_{i,j}^{n+\frac{1}{2}}\right]\nonumber\\
    &+\frac{\mu_y\sqrt{M_{i,j}^{n+\frac{1}{2}}}}{2}\left[\left(\frac{\rho_{i,j-1}^{n}}{\sqrt{M_{i,j-1}^{n+\frac{1}{2}}}}\right)+\left(\frac{\rho_{i,j+1}^{n}}{\sqrt{M_{i,j+1}^{n+\frac{1}{2}}}}\right)-\left( \frac{\sqrt{M_{i,j-1}^{n+\frac{1}{2}}}+\sqrt{M_{i,j+1}^{n+\frac{1}{2}}}}{M_{i,j}^{n+\frac{1}{2}}}\right)\rho_{i,j}^{n}\right]+\rho_{i,j}^{n}.
\end{align}

\begin{align}\label{eqn:cond1}
    1-\frac{\mu_y}{2}\left( \frac{\sqrt{M_{i,j-1}^{n+\frac{1}{2}}}+\sqrt{M_{i,j+1}^{n+\frac{1}{2}}}}{\sqrt{M_{i,j}^{n+\frac{1}{2}}}}\right)\geq 0
\end{align}
holds for arbitrary $i,j$, then the first-step of the additive ADI \eqref{eqn:additiveADI_rho1} preserves positivity. To prove this, we reformulate \eqref{eqn:2nd-scheme1} as
\begin{align}\label{eqn:2nd-scheme2}
    \rho_{i,j}^{n+\frac{1}{2}}&=\frac{\mu_x}{2}\left[ \sqrt{M_{i-1,j}^{n+\frac{1}{2}}M_{i,j}^{n+\frac{1}{2}}}\left( \frac{\rho_{i-1,j}^{n+\frac{1}{2}}}{M_{i-1,j}^{n+\frac{1}{2}}}-\frac{\rho_{i,j}^{n+\frac{1}{2}}}{M_{i,j}^{n+\frac{1}{2}}}\right)+\sqrt{M_{i+1,j}^{n+\frac{1}{2}}M_{i,j}^{n+\frac{1}{2}}}\left(\frac{\rho_{i+1,j}^{n+\frac{1}{2}}}{M_{i+1,j}^{n+\frac{1}{2}}}-\frac{\rho_{i,j}^{n+\frac{1}{2}}}{M_{i,j}^{n+\frac{1}{2}}}\right) \right]\nonumber\\
    &+\left[1-\frac{\mu_y}{2}\left( \frac{\sqrt{M_{i,j-1}^{n+\frac{1}{2}}}+\sqrt{M_{i,j+1}^{n+\frac{1}{2}}}}{\sqrt{M_{i,j}^{n+\frac{1}{2}}}}\right)\right]\rho_{i,j}^{n}+\frac{\mu_y\sqrt{M_{i,j}^{n+\frac{1}{2}}}}{2\sqrt{M_{i,j-1}^{n+\frac{1}{2}}}}\rho_{i,j-1}^{n}+\frac{\mu_y\sqrt{M_{i,j}^{n+\frac{1}{2}}}}{2\sqrt{M_{i,j+1}^{n+\frac{1}{2}}}}\rho_{i,j+1}^{n}.
\end{align}
We assume that $\rho_{i,j}^{n}$ is nonnegative and the condition \eqref{eqn:cond1} holds for all $i,j$, then the last three terms in Eq. \eqref{eqn:2nd-scheme2} are nonnegative. Using the same argument in the proof of Theorem \ref{Thm1} again, let $(k,l)$ be the indices satisfy $\rho_{k,l}^{n+\frac{1}{2}}/M_{k,l}^{n+\frac{1}{2}}=\min_{i,j}\{\rho_{i,j}^{n+\frac{1}{2}}/M_{i,j}^{n+\frac{1}{2}}\}$. The terms in the first bracket of the RHS of Eq. \eqref{eqn:2nd-scheme2} are nonnegative for $(i,j)=(k,l)$. From Eq. \eqref{eqn:2nd-scheme2} and $M_{k,l}^{n+\frac{1}{2}}\geq 0$, we have $\rho_{k,l}^{n+\frac{1}{2}}\geq 0$, then $\rho_{k,l}^{n+\frac{1}{2}}/M_{k,l}^{n+\frac{1}{2}}\geq 0$. That is, $\rho_{i,j}^{n+\frac{1}{2}}/M_{i,j}^{n+\frac{1}{2}}\geq 0$ for all $i,j$ because $\rho_{k,l}^{n+\frac{1}{2}}/M_{k,l}^{n+\frac{1}{2}}\geq 0$ is the minimum value from the definition.\\ 
Similarly, if 
\begin{align}\label{eqn:cond2}
    1-\frac{\mu_x}{2}\left( \frac{\sqrt{M_{i-1,j}^{n+\frac{1}{2}}}+\sqrt{M_{i+1,j}^{n+\frac{1}{2}}}}{\sqrt{M_{i,j}^{n+\frac{1}{2}}}}\right)\geq 0
\end{align}
holds for arbitrary $i,j$, then the second-step of the additive ADI \eqref{eqn:additiveADI_rho2} preserves positivity. That is, the scheme \eqref{eqn:additiveADI_rho} preserves positivity under these two inequality conditions \eqref{eqn:cond1} and \eqref{eqn:cond2}.
\begin{remark}
    We have proved this additive ADI scheme preserving positivity and it is straight forward to verify the mass conservation similar to the previous section. However, we cannot verify the energy dissipation law for the fully discrete scheme \eqref{eqn:additiveADI_c} and \eqref{eqn:additiveADI_rho} since it is hard to deal with $\rho_{i,j}^n/\sqrt{M_{i,j}^{n+\frac{1}{2}}}$ and $\rho_{i,j}^{n+1}/\sqrt{M_{i,j}^{n+\frac{1}{2}}}$.
\end{remark}
\section{ Numerical experiments}
The solution to Keller-Segel equations have several important properties, such as positivity preservation, mass preservation, asymptotic behavior. Furthermore, the solution blows up when the initial mass is great than some critical value. In the following, we verify the order of convergence for two ADI schemes first. Then, we compare the computational time  of our ADI schemes \eqref{eqn:semi-discrete4ADI_c2} and \eqref{eqn:semi-discrete4ADI_rho2} with the standard five-point method \eqref{eqn:semi-discrete4ADI_c1} and \eqref{eqn:semi-discrete4ADI_rho1}. As we will see in the following results, the ADI scheme significantly reduces the computational cost.
\subsection{The order of convergence}\label{sec:convergence}
In this subsection, we focus on checking the order of accuracy of the ADI schemes \eqref{eqn:semi-discrete4ADI_c3}-\eqref{eqn:semi-discrete4ADI_c4} and \eqref{eqn:semi-discrete4ADI_rho11}-\eqref{eqn:semi-discrete4ADI_rho12}. In order to verify the order of convergence, we construct an exact solution of the 2D Keller-Segel equations \eqref{eqn:original_rho} and \eqref{eqn:original_c} with added known terms. We start with the following exact solutions
\begin{align}\label{eqn:exact_solu}
    \rho(x,y,t) = 4e^{-(t+x^2+y^2)}, \qquad c(x,y,t)=e^{-(t+\frac{x^2+y^2}{2})},
\end{align}
in the square domain $\Omega=(-1,1)\times(-1,1)$.

Denoting 
\begin{align}
    F_1(x,y,t)&=\left[c(x,y,t)\left(3x^2+3y^2-2\right)-4\left(x^2+y^2\right)+3\right]\rho(x,y,t), \label{eqn:F1}\\
    F_2(x,y,t)&=\left(2-\varepsilon-x^2-y^2\right)c(x,y,t)-\rho(x,y,t), \label{eqn:F2}
\end{align}
the solution in Eq. \eqref{eqn:exact_solu} satisfies the following modified Keller-Segel system
\begin{align}\label{eqn:forcing_system}
    \partial_t\rho=\Delta\rho-\nabla\cdot\left(\rho\nabla c\right)+F_1,\qquad    
    \varepsilon\partial_t c=\Delta c+\rho+F_2,
\end{align}
where we choose $\varepsilon=1$ in this subsection.

To investigate the spatial convergence order, we fix time step $\Delta t=10^{-6}$ but vary the spatial grid sizes $\Delta x=\Delta y=0.1, 0.05, 0.025,0.0125$. We examine the error in maximum norm at the output time $T=10^{-5}$.
\begin{equation}
    \text{error}_{\Delta x}=\| f_{\Delta x}(T)-f_{exact}(T)\|_{\infty}:=\max\limits_{i,j}|f_{\Delta x}(x_i,y_j,T)-f_{exact}(x_i,y_j,T)|,
\end{equation}
where $f_{\Delta x}$ and $f_{exact}$ denote the numerical solution with mesh size $\Delta x$ and the exact solution respectively.

To investigate the time convergence order, we fix the grid size in space $\Delta x=\Delta y=0.001$ and compare the errors with different time steps $\Delta t=0.05, 0.025, 0.0125, 0.00625$, and the output time is $T=0.1$.
\begin{equation}
    \text{error}_{\Delta t}=\| f_{\Delta t}(T)-f_{exact}(T)\|_{\infty}:=\max\limits_{i,j}|f_{\Delta t}(x_i,y_j,T)-f_{exact}(x_i,y_j,T)|,
\end{equation}
where $f_{\Delta t}$ denotes the numerical solution with time step $\Delta t$.\\
The errors with different mesh sizes and the numerical order of convergence are presented in Table \ref{table:ADIdensity_fixed_t}. The numerical order is calculated as $\log_2{(\text{error}_{2\Delta x}/\text{error}_{\Delta x})}$ and the results show second-order accuracy in space. Table \ref{table:ADIdensity_fixed_x} shows the errors with different time step sizes and the numerical order of convergence in time. Similarly, the numerical order is calculated as $\log_2{(\text{error}_{2\Delta t}/\text{error}_{\Delta t})}$ and the results confirm the first-order accuracy in time.

\begin{table}
  \centering
  \caption{The spatial convergence order of the ADI scheme Eqs. \eqref{eqn:semi-discrete4ADI_c3}-\eqref{eqn:semi-discrete4ADI_c4} and \eqref{eqn:semi-discrete4ADI_rho11}-\eqref{eqn:semi-discrete4ADI_rho12}: Errors of density $\rho$ and $c$ for different mesh sizes.}
  \label{table:ADIdensity_fixed_t}
  \begin{tabular}{ccccc}
    \toprule
    $\Delta x=\Delta y$ & Maximum error in $\rho$ & Order & Maximum error in $c$ & Order \\
    \midrule
    0.1 & 2.1261E-07 & -- & 4.9951E-08 & -- \\
    0.05 & 5.3292E-08 & 1.9962 &1.2530E-08 & 1.9951\\
    0.025 & 1.3335E-08 & 1.9987 &3.1596E-09 &  1.9876\\
    0.0125 & 3.3376E-09 & 1.9983 & 8.1621E-10 & 1.9527\\
    \bottomrule
  \end{tabular}
\end{table}

\begin{table}
  \centering
  \caption{The time convergence order of the ADI scheme Eqs. \eqref{eqn:semi-discrete4ADI_c3}-\eqref{eqn:semi-discrete4ADI_c4} and \eqref{eqn:semi-discrete4ADI_rho11}-\eqref{eqn:semi-discrete4ADI_rho12}: Errors of density $\rho$ and $c$ for different time step sizes.}
\label{table:ADIdensity_fixed_x}
  \begin{tabular}{ccccc}
    \toprule
    $\Delta t$ & Maximum error in $\rho$ & Order & Maximum error in $c$ & Order \\
    \midrule
    0.05 & 0.0093 & -- & 0.0133 & -- \\
    0.025 & 0.0043 & 1.1129 &0.0070 & 0.9260\\
    0.0125 & 0.0021 & 1.0339 &0.0036 &  0.9593\\
    0.00625 & 9.9789E-04 & 1.0734 & 0.0018 & 1.0000\\
    \bottomrule
  \end{tabular}
\end{table}

\subsection{Second-order convergence in time}\label{sec:2nd-order in time}
In this subsection, we follow the same example presented in Sec. \ref{sec:convergence} to verify the order of accuracy of the additive ADI scheme \eqref{eqn:additiveADI_c1}-\eqref{eqn:additiveADI_c2} and \eqref{eqn:additiveADI_rho1}-\eqref{eqn:additiveADI_rho2}.\\
Our goal is to verify the additive ADI scheme for the following equations: 
\begin{align}
    \partial_t\rho&=\Delta\rho-\nabla\cdot(\rho\nabla c)+F_1=\nabla\cdot\left(M\nabla\left(\frac{\rho}{M}\right)\right)+F_1,\\    
    \varepsilon\partial_tc&=\Delta c+\rho+F_2,
\end{align}
where $M=e^c$. The semi-discrete scheme is 
\begin{align}
    \varepsilon\frac{c^{n+\frac{1}{2}}-c^n}{\Delta t/2}&=\frac{1}{(\Delta x)^2}\delta_{x}^2c^{n+\frac{1}{2}}+\frac{1}{(\Delta y)^2}\delta_{y}^2c^{n}+\rho^{n}+F_{2}^{n},\\
    \varepsilon\frac{c^{n+1}-c^{n+\frac{1}{2}}}{\Delta t/2}&=\frac{1}{(\Delta x)^2}\delta_{x}^2c^{n+\frac{1}{2}}+\frac{1}{(\Delta y)^2}\delta_{y}^2c^{n+1}+\rho^{n+1}+F_{2}^{n+1},
\end{align}
and 
\begin{align}
     \frac{\rho^{n+\frac{1}{2}}-\rho^n}{\Delta t/2}&=\sqrt{M^{n+\frac{1}{2}}}\left[\frac{1}{(\Delta x)^2}\bar{\tau}_{x}\left(\frac{\rho^{n+\frac{1}{2}}}{\sqrt{M^{n+\frac{1}{2}}}}\right)+\frac{1}{(\Delta y)^2}\bar{\tau}_{y}\left(\frac{\rho^{n}}{\sqrt{M^{n+\frac{1}{2}}}}\right)\right]+F_{1}^{n},\\
     \frac{\rho^{n+1}-\rho^{n+\frac{1}{2}}}{\Delta t/2}&=\sqrt{M^{n+\frac{1}{2}}}\left[\frac{1}{(\Delta x)^2}\bar{\tau}_{x}\left(\frac{\rho^{n+\frac{1}{2}}}{\sqrt{M^{n+\frac{1}{2}}}}\right)+\frac{1}{(\Delta y)^2}\bar{\tau}_{y}\left(\frac{\rho^{n+1}}{\sqrt{M^{n+\frac{1}{2}}}}\right)\right]+F_{1}^{n+1}.
\end{align}
The semi-discrete scheme is equivalent to 
\begin{align}
    \varepsilon c^{n+\frac{1}{2}}-\frac{\mu_x}{2}\delta_{x}^{2}c^{n+\frac{1}{2}}&=\varepsilon c^n+\frac{\mu_y}{2}\delta_{y}^{2}c^n+\frac{\Delta t}{2}(\rho^n+F_{2}^{n}),\nonumber\\
    \varepsilon c^{n+1}-\frac{\mu_y}{2}\delta_{y}^{2}c^{n+1}&=\varepsilon c^{n+\frac{1}{2}}+\frac{\mu_x}{2}\delta_{x}^{2}c^{n+\frac{1}{2}}+\frac{\Delta t}{2}(\rho^{n+1}+F_{2}^{n+1}),
\end{align}
and 
\begin{align}
    \rho^{n+\frac{1}{2}}-\frac{\mu_x}{2}\sqrt{M^{n+\frac{1}{2}}}\bar{\tau}_x\left( \frac{\rho^{n+\frac{1}{2}}}{\sqrt{M^{n+\frac{1}{2}}}} \right)&= \rho^{n}+\frac{\mu_y}{2}\sqrt{M^{n+\frac{1}{2}}}\bar{\tau}_y\left( \frac{\rho^{n}}{\sqrt{M^{n+\frac{1}{2}}}} \right)+\frac{\Delta t}{2}F_{1}^{n},\nonumber\\
    \rho^{n+1}-\frac{\mu_y}{2}\sqrt{M^{n+\frac{1}{2}}}\bar{\tau}_y\left( \frac{\rho^{n+1}}{\sqrt{M^{n+\frac{1}{2}}}} \right)&= \rho^{n+\frac{1}{2}}+\frac{\mu_x}{2}\sqrt{M^{n+\frac{1}{2}}}\bar{\tau}_x\left( \frac{\rho^{n+\frac{1}{2}}}{\sqrt{M^{n+\frac{1}{2}}}} \right)+\frac{\Delta t}{2}F_{1}^{n+1}.
\end{align}
To investigate the time convergence order, we fix the grid size in space $\Delta x=\Delta y=0.001$ and compare the errors with different time steps $\Delta t=0.01, 0.005, 0.0025, 0.00125$, and the output time is $T=0.04$. We choose the parameter $\varepsilon=1$ in this case. Table \ref{table:additiveADIdensity_fixed_x} shows the errors with different time step sizes, which confirms the second-order accuracy in time.

\begin{table}
  \centering
  \caption{The time convergence order of the additive ADI scheme Eqs. \eqref{eqn:additiveADI_c1}-\eqref{eqn:additiveADI_c2} and \eqref{eqn:additiveADI_rho1}-\eqref{eqn:additiveADI_rho2}: Errors of density $\rho$ and $c$ for different time step sizes.}
\label{table:additiveADIdensity_fixed_x}
  \begin{tabular}{ccccc}
    \toprule
    $\Delta t$ & Maximum error in $\rho$ & Order & Maximum error in $c$ & Order \\
    \midrule
    0.01 & 4.7003E-05 & -- & 1.2824E-05 & -- \\
    0.005 & 1.1076E-05& 2.0853 &3.1129E-06 & 2.0425\\
    0.0025 & 2.5185E-06 & 2.1368 &7.2538E-07 &  2.1014\\
    0.00125 & 6.6010E-07 & 1.9318 & 1.5921E-07 & 2.1878\\
    \bottomrule
  \end{tabular}
\end{table}

\textbf{Stability.}  To numerically investigate the stability of the 2nd-order ADI scheme \eqref{eqn:additiveADI_c1}-\eqref{eqn:additiveADI_rho2}, we compute the solution to the modified Keller-Segel system \eqref{eqn:forcing_system}, the same example as in Section \ref{sec:convergence}. We summarize the relative $L^2$ errors of the density $\rho$ at final time $T=5$ computed by the 2nd-order ADI scheme \eqref{eqn:additiveADI_c1}-\eqref{eqn:additiveADI_rho2} with various mesh sizes and time step sizes in Table \ref{table:2nd-stability_rho}. The numerical results show that the errors mainly come from the time integration for the range of grid sizes and time step sizes. The results are stable for the time step size $\Delta t$ as large as $0.2$ and the grid sizes $\Delta x$ and $\Delta y$  as small as $0.002$, suggesting that the 2nd-order ADI scheme is stable. We obtain similar results for the concentration $c$.  The extrapolation $2\rho^n-\rho^{n-1}$ used in the 2nd-order ADI scheme does not affect the stability. In addition, the first-order semi-discrete scheme \eqref{eqn:semi-discrete_rho} and \eqref{eqn:semi-discrete_c} is shown to be stable for small initial data and $\Delta t \|\nabla  \rho^n \|< 1$ in \cite{Finitedifference}. The numerical stability is verified for the 1st-order ADI scheme \eqref{eqn:semi-discrete4ADI_c2} and \eqref{eqn:semi-discrete4ADI_rho2} from similar tests as above.  

\begin{table}
  \centering
  \caption{The relative $L^2$ errors of the density $\rho$ at final time $T=5$ with various mesh sizes and time step sizes for the 2nd-order ADI scheme \eqref{eqn:additiveADI_c1}-\eqref{eqn:additiveADI_rho2}.}
  \label{table:2nd-stability_rho}
  \begin{tabular}{ccccc}
    \toprule
    $\Delta t$ & $\Delta x=\Delta y=0.05$ & $\Delta x=\Delta y=0.02$ & $\Delta x=\Delta y=0.01$ & $\Delta x=\Delta y=0.002$\\
    \midrule
    0.2 &  0.0080 &  0.0084 & 0.0085 & 0.0086 \\
    0.1 &  0.0017 & 0.0020 & 0.0021 & 0.0021\\
    0.05 & 4.7422E-04 & 4.6599E-04 &5.0646E-04 & 5.2215E-04\\
    0.01 & 5.6766E-04 & 8.1606E-05 &1.9401E-05 & 2.0315E-05\\
    \bottomrule
  \end{tabular}
\end{table}

\subsection{Efficiency}
Blow-up is one of the most important properties of Keller-Segel system \cite{review1,review2}. The solutions of Keller-Segel equations blow up in finite time if the initial mass is large than a critical value. In order to investigate the solutions close to the blow-up time, extremely fine grid is needed. Developing a highly efficient numerical scheme becomes important to reduce computational cost. The original five-point method need to solve \eqref{eqn:semi-discrete4ADI_c1} and \eqref{eqn:semi-discrete4ADI_rho1} require solving four sparse $N_xN_y\times N_xN_y$ linear systems using an iterative method such as conjugate gradient method. It still results in relatively high computational cost, especially in the case of higher dimensions. It is well known that ADI schemes are a type of fast direct method to solve 2D or higher dimensional parabolic partial differential equations. For the first-step of the ADI scheme for the density \eqref{eqn:semi-discrete4ADI_rho11}, we only need to solve $N_y$ decoupled tridiagonal linear systems of size $N_x\times N_x$. The computational complexity is similar for the rest substeps of the ADI scheme Eqs. \eqref{eqn:semi-discrete4ADI_rho12}, \eqref{eqn:semi-discrete4ADI_c3} and \eqref{eqn:semi-discrete4ADI_c4}. In this subsection, we compare the computational efficiency of our ADI scheme with that of the five-point scheme. Here we use the conjugate gradient method to solve the linear system for the original five-point scheme. \\
To compare the computational costs, we use the example in Subsection \ref{sec:convergence} with known exact solutions. We take domain $\Omega=[-5,5]\times[-5,5]$, time step $\Delta t=0.001$, and the final time is $T=1$, and  different space grid sizes $N_x=N_y=80,160,320,640$ respectively. The running times of our ADI scheme \eqref{eqn:semi-discrete4ADI_c3}-\eqref{eqn:semi-discrete4ADI_c4} and \eqref{eqn:semi-discrete4ADI_rho11}-\eqref{eqn:semi-discrete4ADI_rho12} and the original method \eqref{eqn:semi-discrete4ADI_c1}, \eqref{eqn:semi-discrete4ADI_rho1} are presented in Table \ref{table:efficiency}. The execution times are measured using MATLAB scripts on a personal computer with Intel(R) Core(TM) i5-8265U CPU @ 1.60GHz 1.80GHz and 8 GB RAM. The ADI scheme reduces the computational cost dramatically, as we expect. Both running times increase linearly with the number of unknowns, $N=2N_xN_y$ until the original scheme runs into the limitation of the PC memory. The ADI scheme is about seven times faster than that of the original one. It is worth pointing out the linear systems corresponding to the ADI scheme are solved directly and the solutions to the linear systems are accurate up to round-off errors, while the tolerance of the CG iterative solver is set to be $10^{-10}$ thus the numerical results from the original scheme are less accurate. Therefore, the ADI scheme could be used to capture more accurate asymptotic behavior near the blow-up of the solutions.

\begin{table}
  \centering
  \caption{Comparison between the running times in (seconds) of the ADI scheme \eqref{eqn:semi-discrete4ADI_c3}-\eqref{eqn:semi-discrete4ADI_c4} and \eqref{eqn:semi-discrete4ADI_rho11}-\eqref{eqn:semi-discrete4ADI_rho12} and those of the original five-point scheme \eqref{eqn:semi-discrete4ADI_c1} and \eqref{eqn:semi-discrete4ADI_rho1} (ORIG for short) with different mesh sizes. $N$ denotes the total number of unknowns in the system.}
\label{table:efficiency}
  \begin{tabular}{ccc}
    \toprule
    $N$ & ADI & ORIG  \\
    \midrule
    $2\cdot80^2$ & 3.34s & 20.34s\\
    $2\cdot160^2$ & 10.55s& 73.60s \\
    $2\cdot320^2$ & 38.05s & 296.76s \\
    $2\cdot640^2$ & 145.69s & 1530.91s \\
    \bottomrule
  \end{tabular}
\end{table}

\subsection{An illustrative example}
To demonstrate the conserved properties of the numerical schemes, such as the mass conservation and the energy dissipation law, in this subsection we present the evolution of the relevant quantities by solving the 2D Keller-Segel equations \eqref{eqn:original_rho} and \eqref{eqn:original_c} with periodic or Neumann-type boundary conditions. 

Consider the original Keller-Segel equations \eqref{eqn:original_rho} and \eqref{eqn:original_c} with the following initial conditions:
\begin{align}
    \rho(x,y,0) = 50 e^{-60(x^2+y^2)},\qquad
    c(x,y,0) = 50 e^{-30(x^2+y^2)}.
\label{eqn:icie1}
\end{align}
The nonnegative constant is chosen as $\varepsilon=1$ and the spatial domain $\Omega=(-1,1)\times (-1,1)$. We set the mesh sizes $\Delta x=\Delta y=0.02$ and the time step size $\Delta t=0.0001$. In order to verify the proposed ADI scheme \eqref{eqn:semi-discrete4ADI_c2}, \eqref{eqn:semi-discrete4ADI_rho2} is positivity-preserving, we plot the evolution of the quantities $\min\limits_{i,j}\{\rho_{i,j}\}$ and $\min\limits_{i,j}\{c_{i,j}\}$ in Figure \ref{fig:minVariables}. The quantities that lead to the conservation of mass defined in \eqref{eqn:mass-tau_x}, \eqref{eqn:mass-tau_xy} together with the total mass $\rho_{\text{tot}}(t) := \int_{\Omega}\rho\, dxdy$ and $c_{\text{tot}}(t) =  \int_{\Omega}c\, dxdy$ for periodic or the vanishing Neumann boundary conditions are also shown in Figure \ref{fig:tau_ADI} and Figure \ref{fig:totalmass}, respectively. 
The numerical results show that the mass is conserved up to the round-off level up to the time $t=2$. Note that the total mass of $c$, $c_{\text{tot}}(t)$ increases linearly in time according to $ c_{\text{tot}}(t) = c_{\text{tot}}(0) + \rho_{\text{tot}} (0) t$. In Figure~\ref{fig:energy_ADI_periodic}, we show the evolution of the discrete free energy $\mathcal{E}^n$ defined by \eqref{eqn:fully_discrete_energy_1} for this case with periodic boundary condition and compare the energy decay rate $(\mathcal{E}^{n+1}-\mathcal{E}^n)/\Delta t$ and the RHS of \eqref{eqn:fullyenergy} divided by $\Delta t$. The numerical results show that the inequality \eqref{eqn:fullyenergy} holds all the time. The numerical results for vanishing Neumann boundary conditions are similar as shown in Figure~\ref{fig:energy_ADI_Neumann}.

\begin{figure}
\centering
\includegraphics[trim={0cm 0cm 0cm 0cm},clip,width=\textwidth]{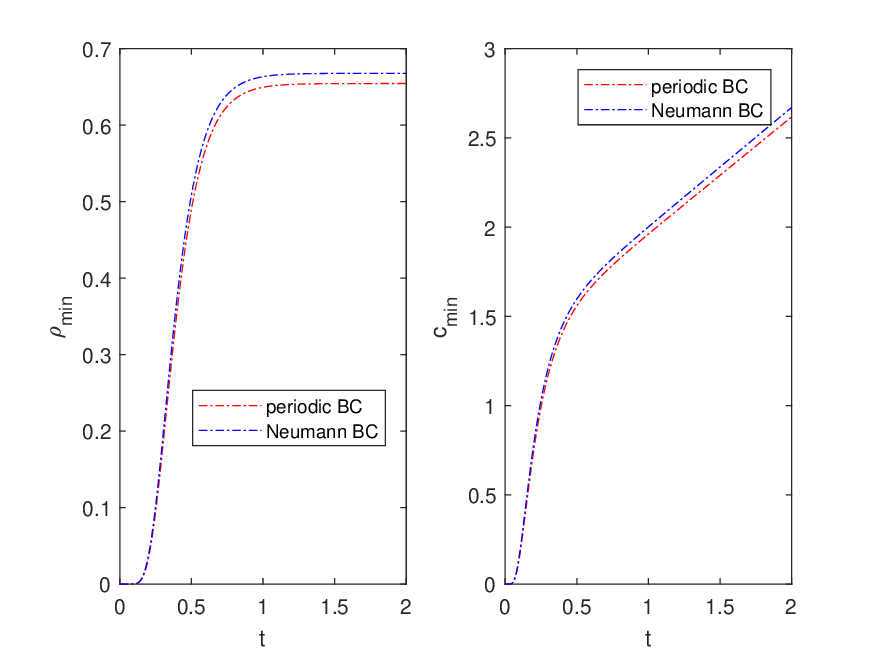}
\caption{The evolution of the minimum values of the density $\rho$ and the concentration $c$ for the solution of the Keller-Segel equations \eqref{eqn:original_rho} and \eqref{eqn:original_c} with $\epsilon=1$, the initial condition \eqref{eqn:icie1} and the periodic or vanishing Neumann boundary conditions.  (Left) $\rho_{\text{min}}:=\min\limits_{i,j}\{\rho_{i,j}\}$. (Right) $c_{\text{min}}:=\min\limits_{i,j}\{c_{i,j}\}$. }
\label{fig:minVariables}
\end{figure}

\begin{figure}
\centering
\includegraphics[trim={0cm 0cm 0cm 0cm},clip,width=\textwidth]{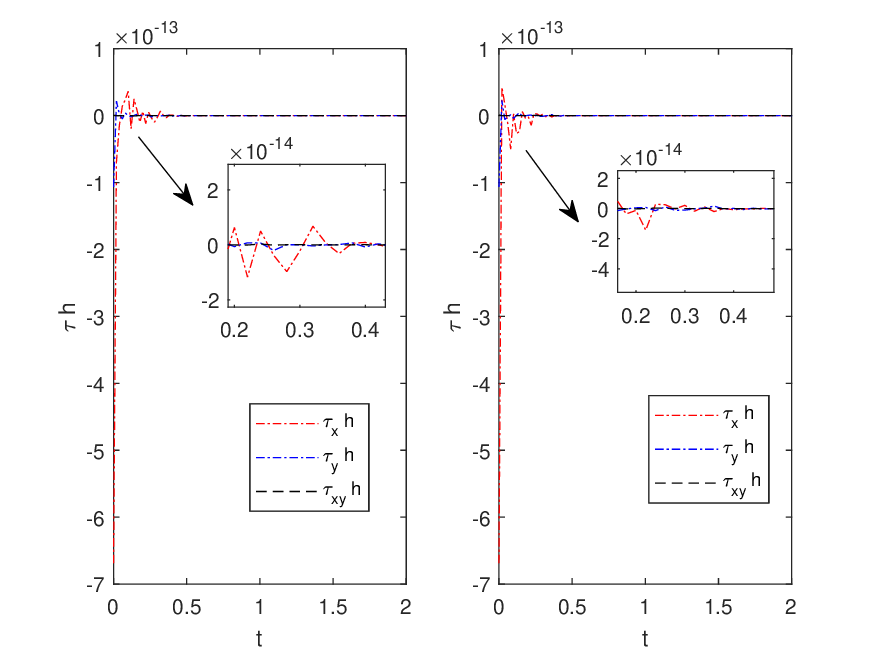}
\caption{The quantities defined in \eqref{eqn:mass-tau_x} and \eqref{eqn:mass-tau_xy} for the solution of the Keller-Segel equations \eqref{eqn:original_rho} and \eqref{eqn:original_c} with $\epsilon=1$ and the initial condition \eqref{eqn:icie1}. (Left) Periodic boundary conditions. (Right) Vanishing Neumann boundary conditions.}
\label{fig:tau_ADI}
\end{figure}

\begin{figure}
\centering
\includegraphics[trim={0cm 0cm 0cm 0cm},clip,width=\textwidth]{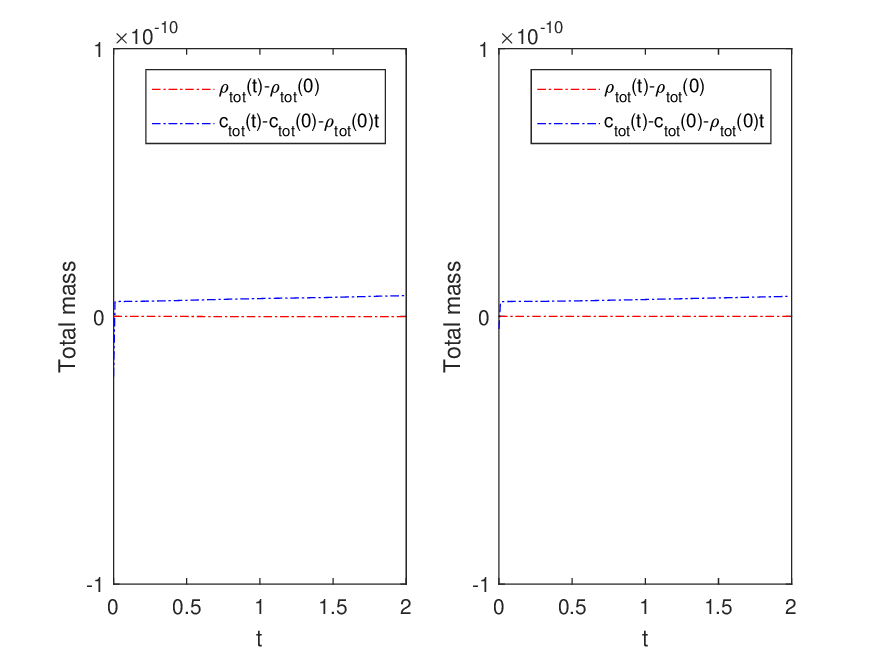}
\caption{The evolution of the total mass compared with the initial values, namely $\rho_{\text{tot}} (t) - \rho_{\text{tot}}(0)$ and $c_{\text{tot}} (t) - c_{\text{tot}}(0) - \rho_{\text{tot}}(0) t$,  for the solution of the Keller-Segel equations \eqref{eqn:original_rho} and \eqref{eqn:original_c} with $\epsilon=1$ and the initial condition \eqref{eqn:icie1}.  (Left) Periodic boundary conditions. (Right) Vanishing Neumann boundary conditions.}
\label{fig:totalmass}
\end{figure}

\begin{figure}
\centering
\includegraphics[trim={0cm 0cm 0cm 0cm},clip,width=\textwidth]{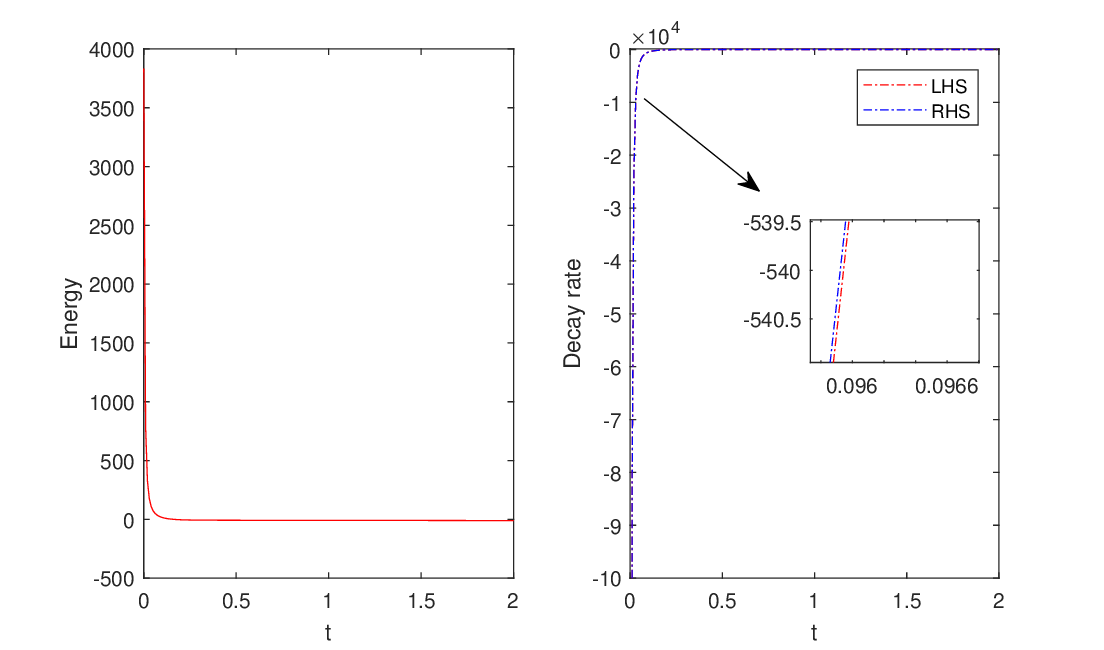}
\caption{(Left) The evolution of the discrete free energy $\mathcal{E}^n$  \eqref{eqn:fully_discrete_energy_1} for the solution of the Keller-Segel equations \eqref{eqn:original_rho} and \eqref{eqn:original_c} with $\epsilon=1$ and the initial condition \eqref{eqn:icie1} with periodic boundary conditions. (Right) The corresponding the energy decay rate $(\mathcal{E}^{n+1}-\mathcal{E}^n)/\Delta t$ and the RHS of \eqref{eqn:fullyenergy} divided by $\Delta t$.}
\label{fig:energy_ADI_periodic}
\end{figure}

\begin{figure}
\centering
\includegraphics[trim={0cm 0cm 0cm 0cm},clip,width=\textwidth]{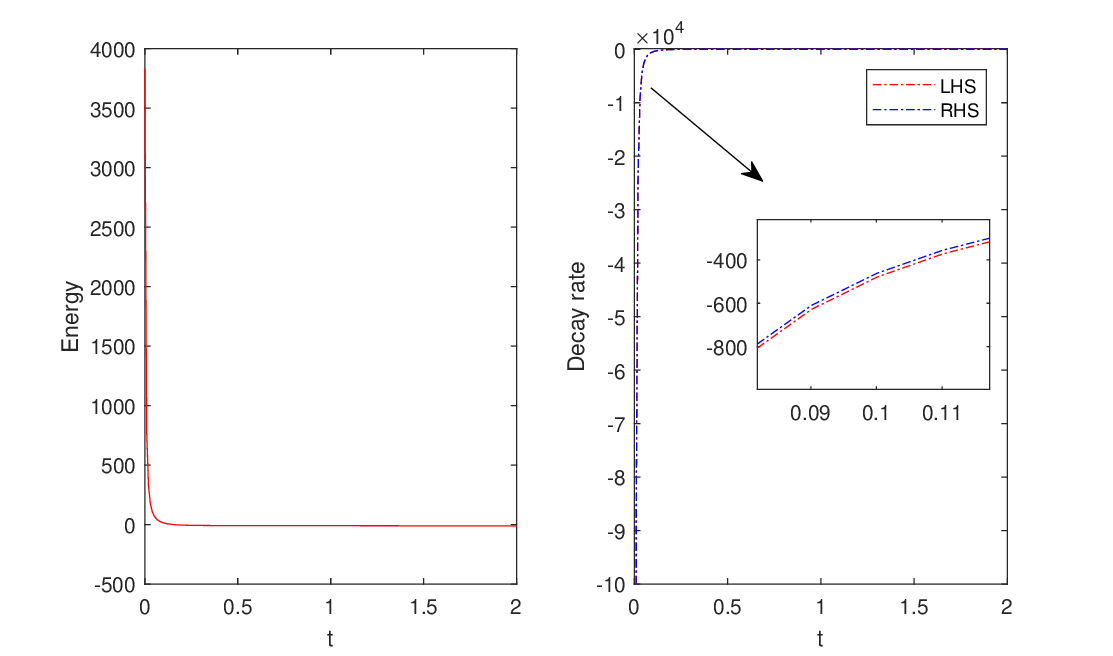}
\caption{(Left) The evolution of the discrete free energy $\mathcal{E}^n$  \eqref{eqn:fully_discrete_energy_1} for the solution of the Keller-Segel equations \eqref{eqn:original_rho} and \eqref{eqn:original_c} with $\epsilon=1$ and the initial condition \eqref{eqn:icie1} with vanishing Neumann boundary conditions. (Right) The corresponding the energy decay rate $(\mathcal{E}^{n+1}-\mathcal{E}^n)/\Delta t$ and the RHS of \eqref{eqn:fullyenergy} divided by $\Delta t$. }
\label{fig:energy_ADI_Neumann}
\end{figure}

\section{Conclusion}
We have presented two efficient finite difference schemes for solving the 2D Keller-Segel equations based on ADI method. Both schemes achieve second-order accuracy in space. One of the schemes exhibits first-order accuracy in time while unconditionally preserving positivity and the energy dissipation law asymptotically. The other scheme achieves second-order accuracy in time, but it only preserves positivity under certain conditions. Furthermore, we demonstrate that our schemes outperform the standard five-point scheme in terms of computational cost.

There are still important aspects that require further investigation. First, although the additive ADI scheme attains second-order accuracy in time, its ability to preserve positivity is conditional. This limitation hinders its widespread adoption. Second, the energy dissipation law for the first-order ADI scheme holds asymptotically rather than exactly. Third, we can deduce the Keller-Segel equations from the respective free energy using the energetic variational approach. This approach offers an alternative method for investigating the blow-up solutions of the Keller-Segel equations, known as the Lagrangian scheme. Additionally, we employ the semi-implicit scheme to decouple the two nonlinear partial equations. However, it is important to note that the explicit components of the semi-implicit scheme may introduce instability issues when dealing with large initial mass conditions. Further investigation is warranted in this regard.

\section*{Acknowledgments}
This work is partially supported by DOE DE-SC0022276. C Liu's research is partially supported by NSF grants DMS-1950868, DMS-2153029 and DMS-2118181. Lu would like to thank Yiwei Wang for helpful discussions.


\setcounter{equation}{0}
\renewcommand{\theequation}{\Alph{section}.\arabic{equation}}
\newpage
\appendix
\markboth{Appendix}{Appendix}
\onecolumn

\section{Detailed Expressions for the Operators $\tau_x,\tau_y$ and $\tau_x\tau_y$}\label{APP}

In order to derive the truncation errors and prove the positivity preservation and mass conservation, we expand the operators $\tau_x,\tau_y$ and $\tau_x\tau_y$ fully
\begin{align}\label{eqn:taux}
    \tau_x h_{i,j}^{n+1}:=&\frac{1}{\sqrt{M_{i,j}^{n+1}}}\delta_x\left(M_{i,j}^{n+1}\delta_x\left(\frac{h^{n+1}}{\sqrt{M^{n+1}}}\right)_{i,j}\right)\nonumber\\
    =&\frac{1}{\sqrt{M_{i,j}^{n+1}}}\left[\sqrt{M_{i,j}^{n+1}M_{i+1,j}^{n+1}}\left(\frac{h_{i+1,j}^{n+1}}{\sqrt{M_{i+1,j}^{n+1}}}-\frac{h_{i,j}^{n+1}}{\sqrt{M_{i,j}^{n+1}}}\right) - \sqrt{M_{i-1,j}^{n+1}M_{i,j}^{n+1}}\left(\frac{h_{i,j}^{n+1}}{\sqrt{M_{i,j}^{n+1}}}-\frac{h_{i-1,j}^{n+1}}{\sqrt{M_{i-1,j}^{n+1}}}\right)\right].
\end{align}
Similarly, we have
\begin{align}\label{eqn:tauy}
    \tau_y h_{i,j}^{n+1}:=&\frac{1}{\sqrt{M_{i,j}^{n+1}}}\delta_y\left(M_{i,j}^{n+1}\delta_y\left(\frac{h_{i,j}^{n+1}}{\sqrt{M_{i,j}^{n+1}}}\right)\right)\nonumber\\
    =&\frac{1}{\sqrt{M_{i,j}^{n+1}}}\left[\sqrt{M_{i,j}^{n+1}M_{i,j+1}^{n+1}}\left(\frac{h_{i,j+1}^{n+1}}{\sqrt{M_{i,j+1}^{n+1}}}-\frac{h_{i,j}^{n+1}}{\sqrt{M_{i,j}^{n+1}}}\right) - \sqrt{M_{i,j-1}^{n+1}M_{i,j}^{n+1}}\left(\frac{h_{i,j}^{n+1}}{\sqrt{M_{i,j}^{n+1}}}-\frac{h_{i,j-1}^{n+1}}{\sqrt{M_{i,j-1}^{n+1}}}\right)\right].
\end{align}
For the composite operator $\tau_x\tau_y$, we have
\begin{align}\label{eqn:tauxy}
    \tau_x\tau_y h_{i,j}^{n+1}=&\frac{1}{\sqrt{M_{i,j}^{n+1}}}\delta_x\left[M_{i,j}^{n+1}\delta_x\left(\frac{\tau_y h}{\sqrt{M}}\right)_{i,j}^{{\color{blue}n+1}}\right]\nonumber\\
    =&\frac{M_{i+\frac{1}{2},j}^{n+1}}{\sqrt{M_{i,j}^{n+1}}}\left[\left(\frac{\tau_yh}{\sqrt{M}}\right)_{i+1,j}^{n+1}-\left(\frac{\tau_yh}{\sqrt{M}}\right)_{i,j}^{n+1}\right]\nonumber\\
    &-\frac{M_{i-\frac{1}{2},j}^{n+1}}{\sqrt{M_{i,j}^{n+1}}}\left[\left(\frac{\tau_yh}{\sqrt{M}}\right)_{i,j}^{n+1}-\left(\frac{\tau_yh}{\sqrt{M}}\right)_{i-1,j}^{n+1}\right],
\end{align}
where the first term on the right-hand side is 
\begin{align}\label{eqn:part1_composite_term}
    &\frac{M_{i+\frac{1}{2},j}^{n+1}}{\sqrt{M_{i,j}^{n+1}}}\left[\left(\frac{\tau_yh}{\sqrt{M}}\right)_{i+1,j}^{n+1}-\left(\frac{\tau_yh}{\sqrt{M}}\right)_{i,j}^{n+1}\right]\nonumber\\
    =&\sqrt{M_{i+1,j+1}^{n+1}}\left[\left(\frac{h}{\sqrt{M}}\right)_{i+1,j+1}^{n+1}-\left(\frac{h}{\sqrt{M}}\right)_{i+1,j}^{n+1}\right]\nonumber\\
    &-\sqrt{M_{i+1,j-1}^{n+1}}\left[\left(\frac{h}{\sqrt{M}}\right)_{i+1,j}^{n+1}-\left(\frac{h}{\sqrt{M}}\right)_{i+1,j-1}^{n+1}\right]\nonumber\\
    &-\frac{\sqrt{M_{i+1,j}^{n+1}M_{i,j+1}^{n+1}}}{\sqrt{M_{i,j}^{n+1}}}\left[\left(\frac{h}{\sqrt{M}}\right)_{i,j+1}^{n+1}-\left(\frac{h}{\sqrt{M}}\right)_{i,j}^{n+1}\right]\nonumber\\
    &+\frac{\sqrt{M_{i+1,j}^{n+1}M_{i,j-1}^{n+1}}}{\sqrt{M_{i,j}^{n+1}}}\left[\left(\frac{h}{\sqrt{M}}\right)_{i,j}^{n+1}-\left(\frac{h}{\sqrt{M}}\right)_{i,j-1}^{n+1}\right],
\end{align}
and the second term on the right-hand side is 
\begin{align}\label{eqn:part2_composite_term}
    &\frac{M_{i-\frac{1}{2},j}^{n+1}}{\sqrt{M_{i,j}^{n+1}}}\left[\left(\frac{\tau_yh}{\sqrt{M}}\right)_{i,j}^{n+1}-\left(\frac{\tau_yh}{\sqrt{M}}\right)_{i-1,j}^{n+1}\right]\nonumber\\
    =&\frac{\sqrt{M_{i-1,j}^{n+1}M_{i,j+1}^{n+1}}}{\sqrt{M_{i,j}^{n+1}}}\left[\left(\frac{h}{\sqrt{M}}\right)_{i,j+1}^{n+1} - \left(\frac{h}{\sqrt{M}}\right)_{i,j}^{n+1}\right]\nonumber\\
    &-\frac{\sqrt{M_{i-1,j}^{n+1}M_{i,j-1}^{n+1}}}{\sqrt{M_{i,j}^{n+1}}}\left[\left(\frac{h}{\sqrt{M}}\right)_{i,j}^{n+1} - \left(\frac{h}{\sqrt{M}}\right)_{i,j-1}^{n+1}\right]\nonumber\\
    &-\sqrt{M_{i-1,j+1}^{n+1}}\left[\left(\frac{h}{\sqrt{M}}\right)_{i-1,j+1}^{n+1} - \left(\frac{h}{\sqrt{M}}\right)_{i-1,j}^{n+1}\right]\nonumber\\
    &+\sqrt{M_{i-1,j-1}^{n+1}}\left[\left(\frac{h}{\sqrt{M}}\right)_{i-1,j}^{n+1} - \left(\frac{h}{\sqrt{M}}\right)_{i-1,j-1}^{n+1}\right].
\end{align}
Here, $M_{i-\frac{1}{2},j}^{n+1}$ is approximated by its geometric mean $\sqrt{M_{i-1,j}M_{i,j}}$ and similarly for the other terms at half grid points.

\section{Truncation errors}\label{errors}
\subsection{Truncation errors of the concentration equation}
Using the Taylor series expansion, we know that
\begin{align}
    \Delta_{+t}v(x,y,t)&=v(x,y,t+\Delta t)-v(x,y,t)\nonumber\\
    &=v_t\Delta t+\frac{1}{2}v_{tt}(\Delta t)^2+\frac{1}{6}v_{ttt}(\Delta t)^3+h.o.t.,
\end{align}
\begin{align}
    \delta_{x}v(x,y,t)&=v(x+\frac{1}{2}\Delta x,y,t)-v(x-\frac{1}{2}\Delta x,y,t)\nonumber\\
    &=v_x(\Delta x)+\frac{1}{24}v_{xxx}(\Delta x)^3+h.o.t.,\\
    \delta_{y}v(x,y,t)&=v(x,y+\frac{1}{2}\Delta y,t)-v(x,y-\frac{1}{2}\Delta y,t)\nonumber\\
    &=v_y(\Delta y)+\frac{1}{24}v_{yyy}(\Delta y)^3+h.o.t.,
\end{align}
and 
\begin{align}
    \delta_{x}^{2}v(x,y,t)&=v(x+\Delta x,y,t)-2v(x,y,t)+v(x-\Delta x,y,t)\nonumber\\
    &=v_{xx}(\Delta x)^2+\frac{1}{12}v_{xxxx}(\Delta x)^4+h.o.t.,\\
    \delta_{y}^{2}v(x,y,t)&=v(x,y+\Delta y,t)-2v(x,y,t)+v(x,y-\Delta y,t)\nonumber\\
    &=v_{yy}(\Delta y)^2+\frac{1}{12}v_{yyyy}(\Delta y)^4+h.o.t.,
\end{align}
where h.o.t. stands for higher-order terms. Therefore, the truncation error $T_{c_1}$ of the original five-point scheme \eqref{eqn:semi-discrete4ADI_c1} at $(x,y,t+\Delta t)$:
\begin{align}
    T_{c_1}(x,y,t+\Delta t):=&\varepsilon\frac{\Delta_{-t}c(x,y,t+\Delta t)}{\Delta t}-\frac{\delta_{x}^{2}c(x,y,t+\Delta t)}{(\Delta x)^2}-\frac{\delta_{y}^{2}c(x,y,t+\Delta t)}{(\Delta y)^2}-\rho(x,y,t)\nonumber\\
    =&\varepsilon\frac{c_t\Delta t-\frac{1}{2}c_{tt}(\Delta t)^2+\frac{1}{6}c_{ttt}(\Delta t)^3}{\Delta t} -\frac{c_{xx}(\Delta x)^2+\frac{1}{12}c_{xxxx}(\Delta x)^4}{(\Delta x)^2}\nonumber\\
    &-\frac{c_{yy}(\Delta y)^2+\frac{1}{12}c_{yyyy}(\Delta y)^4}{(\Delta y)^2}-\rho+\rho_t(\Delta t)-\frac{1}{2}\rho_{tt}(\Delta t)^2+h.o.t.\nonumber\\
    =&(\varepsilon c_t-\Delta c-\rho) -\frac{\varepsilon}{2}c_{tt}(\Delta t)+\frac{\varepsilon}{6}c_{ttt}(\Delta t)^2-\frac{1}{12}c_{xxxx}(\Delta x)^2-\frac{1}{12}c_{yyyy}(\Delta y)^2\nonumber\\
    &+\rho_t(\Delta t)-\frac{1}{2}\rho_{tt}(\Delta t)^2+h.o.t.\nonumber\\
    =&-\frac{\varepsilon}{2}c_{tt}(\Delta t)+\frac{\varepsilon}{6}c_{ttt}(\Delta t)^2-\frac{1}{12}c_{xxxx}(\Delta x)^2-\frac{1}{12}c_{yyyy}(\Delta y)^2+\rho_t(\Delta t)-\frac{1}{2}\rho_{tt}(\Delta t)^2+h.o.t.\nonumber\\
    =&O\left(\Delta t+(\Delta x)^2+(\Delta y)^2\right).
\end{align}
Similarly, the truncation error $T_{c_2}$ of the ADI scheme \eqref{eqn:semi-discrete4ADI_c2} at $(x,y,t+\Delta t)$:
\begin{align}
    T_{c_2}(x,y,t+\Delta t)=\varepsilon\frac{\Delta_{-t}c}{\Delta t}-\frac{\delta_{x}^{2}c}{(\Delta x)^2}-\frac{\delta_{y}^{2}c}{(\Delta y)^2}+\frac{\Delta t\delta_{x}^{2}\delta_{y}^{2}c}{\varepsilon(\Delta x\Delta y)^2}-\rho(x,y,t).
\end{align}
We only need to calculate $\delta_{x}^{2}\delta_{y}^{2}c$. By the definitions of $\delta_{x}^2$ and $\delta_{y}^2$, we have
\begin{align}
    \delta_{x}^{2}\delta_{y}^{2}c=&\delta_{x}^{2}\big[ c(x,y+\Delta y,t+\Delta t)-2c(x,y,t+\Delta t)+c(x,y-\Delta y,t+\Delta t)\big]\nonumber\\
    =&c(x+\Delta x,y+\Delta y,t+\Delta t)-2c(x,y+\Delta y,t+\Delta t)+c(x-\Delta x,y+\Delta y,t+\Delta t)\nonumber\\
    &-2\big[ c(x+\Delta x,y,t+\Delta t)-2c(x,y,t+\Delta t)+c(x-\Delta x,y,t+\Delta t)\big]\nonumber\\
    &+c(x+\Delta x,y-\Delta y,t+\Delta t)-2c(x,y-\Delta y,t+\Delta t)+c(x-\Delta x,y-\Delta y,t+\Delta t)\nonumber\\
    =&c_{xxyy}(\Delta x\Delta y)^2+h.o.t..
\end{align}
Therefore, the truncation error $T_{c_2}$ of the ADI scheme \eqref{eqn:semi-discrete4ADI_c2}
\begin{align}
    T_{c_2}=T_{c_1}+\frac{c_{xxyy}}{\varepsilon}\Delta t+h.o.t..
\end{align}

\subsection{Truncation errors of the density equation}
Taking Taylor expansion at $(x,y,t+\Delta t)$, we have
\begin{align}
    M_{i+\frac{1}{2},j}^{n+1} &= M_{i,j}^{n+1}+\frac{1}{2}\partial_xM_{i,j}^{n+1}(\Delta x)+\frac{1}{8}\partial_{xx}M_{i,j}^{n+1}(\Delta x)^2+h.o.t.,\\
    M_{i-\frac{1}{2},j}^{n+1} &= M_{i,j}^{n+1}-\frac{1}{2}\partial_xM_{i,j}^{n+1}(\Delta x)+\frac{1}{8}\partial_{xx}M_{i,j}^{n+1}(\Delta x)^2+h.o.t.,\\
    \left(\frac{h}{\sqrt{M}}\right)_{i+1,j}^{n+1}&=\left(\frac{h}{\sqrt{M}}\right)_{i,j}^{n+1}+\partial_x\left(\frac{h}{\sqrt{M}}\right)_{i,j}^{n+1}\left(\Delta x\right)+\frac{1}{2}\partial_{xx}\left(\frac{h}{\sqrt{M}}\right)_{i,j}^{n+1}\left(\Delta x\right)^2+h.o.t.,\\
    \left(\frac{h}{\sqrt{M}}\right)_{i-1,j}^{n+1}&=\left(\frac{h}{\sqrt{M}}\right)_{i,j}^{n+1}-\partial_x\left(\frac{h}{\sqrt{M}}\right)_{i,j}^{n+1}\left(\Delta x\right)+\frac{1}{2}\partial_{xx}\left(\frac{h}{\sqrt{M}}\right)_{i,j}^{n+1}\left(\Delta x\right)^2+h.o.t..
\end{align}
Therefore, the leading-order Taylor expansion of $\tau_x$ is
\begin{align}
    \tau_xh_{i,j}^{n+1}=&\frac{1}{\sqrt{M_{i,j}^{n+1}}}\left[M_{i,j}^{n+1}\partial_{xx}\left(\frac{h}{\sqrt{M}}\right)_{i,j}^{n+1}(\Delta x)^2 + \partial_xM_{i,j}^{n+1}\partial_x\left(\frac{h}{\sqrt{M}}\right)_{i,j}^{n+1}(\Delta x)^2 \nonumber\right.\\
    &\left.+ \frac{1}{8}\partial_{xx}M_{i,j}^{n+1}\partial_{xx}\left(\frac{h}{\sqrt{M}}\right)_{i,j}^{n+1}(\Delta x)^4+h.o.t.\right],
\end{align}
and the corresponding expansion of $\tau_y$ is 
\begin{align}
    \tau_yh_{i,j}^{n+1}=&\frac{1}{\sqrt{M_{i,j}^{n+1}}}\left[M_{i,j}^{n+1}\partial_{yy}\left(\frac{h}{\sqrt{M}}\right)_{i,j}^{n+1}(\Delta y)^2 + \partial_yM_{i,j}^{n+1}\partial_y\left(\frac{h}{\sqrt{M}}\right)_{i,j}^{n+1}(\Delta y)^2 \nonumber\right.\\
    &\left.+ \frac{1}{8}\partial_{yy}M_{i,j}^{n+1}\partial_{yy}\left(\frac{h}{\sqrt{M}}\right)_{i,j}^{n+1}(\Delta y)^4+h.o.t.\right].
\end{align}
The truncation error of the original scheme \eqref{eqn:semi-discrete4ADI_rho1} is
\begin{align}
    T_{\rho_{1}}:=&\frac{\Delta_{-t}\rho_{i,j}^{n+1}}{\Delta t}-\frac{\sqrt{M_{i,j}^{n+1}}\tau_xh_{i,j}^{n+1}}{(\Delta x)^2}-\frac{\sqrt{M_{i,j}^{n+1}}\tau_yh_{i,j}^{n+1}}{(\Delta y)^2}\nonumber\\
    =&\partial_t\rho_{i,j}^{n+1}-\frac{1}{2}\partial_{tt}\rho_{i,j}^{n+1}(\Delta t) - M_{i,j}^{n+1}\partial_{xx}\left(\frac{h}{\sqrt{M}}\right)_{i,j}^{n+1} - \partial_xM_{i,j}^{n+1}\partial_x\left(\frac{h}{\sqrt{M}}\right)_{i,j}^{n+1}\nonumber\\
    &- \frac{1}{8}\partial_{xx}M_{i,j}^{n+1}\partial_{xx}\left(\frac{h}{\sqrt{M}}\right)_{i,j}^{n+1}(\Delta x)^2 - M_{i,j}^{n+1}\partial_{yy}\left(\frac{h}{\sqrt{M}}\right)_{i,j}^{n+1} - \partial_yM_{i,j}^{n+1}\partial_y\left(\frac{h}{\sqrt{M}}\right)_{i,j}^{n+1}\nonumber\\
    &- \frac{1}{8}\partial_{yy}M_{i,j}^{n+1}\partial_{yy}\left(\frac{h}{\sqrt{M}}\right)_{i,j}^{n+1}(\Delta y)^2+h.o.t.\nonumber\\
    =&-\frac{1}{2}\partial_{tt}\rho_{i,j}^{n+1}(\Delta t)- \frac{1}{8}\partial_{xx}M_{i,j}^{n+1}\partial_{xx}\left(\frac{h}{\sqrt{M}}\right)_{i,j}^{n+1}(\Delta x)^2 - \frac{1}{8}\partial_{yy}M_{i,j}^{n+1}\partial_{yy}\left(\frac{h}{\sqrt{M}}\right)_{i,j}^{n+1}(\Delta y)^2+h.o.t.\nonumber\\
    =&O\left(\Delta t+(\Delta x)^2+(\Delta y)^2\right).
\end{align}
In order to calculate the truncation error of the ADI scheme \eqref{eqn:semi-discrete4ADI_rho2}, we only need to derive the truncation error due to the composite operator $\tau_x\tau_y$. Notice that the coefficients $M_{i+1,j+1}^{n+1}, M_{i+1,j-1}^{n+1}, M_{i+1,j}^{n+1}, M_{i,j+1}^{n+1}, M_{i,j-1}^{n+1}$ converge to $M_{i,j}^{n+1}$ when $\Delta x$ and $\Delta y$ tend to $0$. Therefore, the composite operator $\tau_x\tau_y$ \eqref{eqn:tauxy} can be simplified to 
\begin{align}
    \tau_x\tau_y h_{i,j}^{n+1}=&\frac{1}{\sqrt{M_{i,j}^{n+1}}}\delta_x\left[M_{i,j}^{n+1}\delta_x\left(\frac{\tau_y h}{\sqrt{M}}\right)_{i,j}^{{\color{blue}n+1}}\right]\nonumber\\
    =&\frac{M_{i+\frac{1}{2},j}^{n+1}}{\sqrt{M_{i,j}^{n+1}}}\left[\left(\frac{\tau_yh}{\sqrt{M}}\right)_{i+1,j}^{n+1}-\left(\frac{\tau_yh}{\sqrt{M}}\right)_{i,j}^{n+1}\right]\nonumber\\
    &-\frac{M_{i-\frac{1}{2},j}^{n+1}}{\sqrt{M_{i,j}^{n+1}}}\left[\left(\frac{\tau_yh}{\sqrt{M}}\right)_{i,j}^{n+1}-\left(\frac{\tau_yh}{\sqrt{M}}\right)_{i-1,j}^{n+1}\right]\nonumber\\
     =&\sqrt{M_{i,j}^{n+1}}\left\{\left[\left(\frac{h}{\sqrt{M}}\right)_{i+1,j+1}^{n+1}-\left(\frac{h}{\sqrt{M}}\right)_{i+1,j}^{n+1}\right] -\left[\left(\frac{h}{\sqrt{M}}\right)_{i+1,j}^{n+1}-\left(\frac{h}{\sqrt{M}}\right)_{i+1,j-1}^{n+1}\right]\nonumber\right.\\
    &\left.-\left[\left(\frac{h}{\sqrt{M}}\right)_{i,j+1}^{n+1}-\left(\frac{h}{\sqrt{M}}\right)_{i,j}^{n+1}\right]+\left[\left(\frac{h}{\sqrt{M}}\right)_{i,j}^{n+1}-\left(\frac{h}{\sqrt{M}}\right)_{i,j-1}^{n+1}\right]\right\}\nonumber\\
    &+\sqrt{M_{i,j}^{n+1}}\left\{\left[\left(\frac{h}{\sqrt{M}}\right)_{i,j+1}^{n+1}-\left(\frac{h}{\sqrt{M}}\right)_{i,j}^{n+1}\right] -\left[\left(\frac{h}{\sqrt{M}}\right)_{i,j}^{n+1}-\left(\frac{h}{\sqrt{M}}\right)_{i,j-1}^{n+1}\right]\nonumber\right.\\
    &\left.-\left[\left(\frac{h}{\sqrt{M}}\right)_{i-1,j+1}^{n+1}-\left(\frac{h}{\sqrt{M}}\right)_{i-1,j}^{n+1}\right]+\left[\left(\frac{h}{\sqrt{M}}\right)_{i-1,j}^{n+1}-\left(\frac{h}{\sqrt{M}}\right)_{i-1,j-1}^{n+1}\right]\right\}+h.o.t.\nonumber\\
    =&O((\Delta x\Delta y)^2).
\end{align}

Therefore, the truncation error due to the composite operator $\tau_x\tau_y$ is 
\begin{align}
    \tau_x\tau_yv=O((\Delta x\Delta y)^2).
\end{align}
In the end, the truncation error of \eqref{eqn:semi-discrete4ADI_rho2} is
\begin{align}
    T_{\rho_2}=T_{\rho_1}+O(\Delta t).
\end{align}

\end{document}